\documentclass{amsart}
\usepackage{graphicx}
\usepackage{amssymb}
\newtheorem{thm}{Theorem}[section]
\newtheorem{cor}[thm]{Corollary}
\newtheorem{lem}[thm]{Lemma}
\newtheorem{prop}[thm]{Proposition}
\theoremstyle{definition}

\theoremstyle{remark}
\newtheorem{rem}[thm]{Remark}
\numberwithin{equation}{section}
\newcommand{\abs}[1]{\left\vert#1\right\vert}
\newcommand{\set}[1]{\left\{#1\right\}}

\newcommand{\eps}{\varepsilon}

\newcommand{\dbar}{\bar\partial}
\newcommand{\ddbar}{\partial\bar\partial}

\DeclareMathOperator{\dom}{Dom}

\DeclareMathOperator{\re}{Re}
\DeclareMathOperator{\im}{Im}
\DeclareMathOperator{\supp}{supp}
\DeclareMathOperator{\dist}{dist}

\DeclareMathOperator{\Tan}{Tan}

\begin{document}

\title[Bounded Plurisubharmonic Functions for Lipschitz Domains in $\mathbb{CP}^n$]{Bounded Plurisubharmonic Exhaustion Functions for Lipschitz Pseudoconvex Domains in $\mathbb{CP}^n$}%
\author{Phillip S. Harrington}%
\address{SCEN 301, 1 University of Arkansas, Fayetteville, AR 72701}%
\email{psharrin@uark.edu}%

%\thanks{The author is partially supported by NSF grant DMS-1002332}
%\subjclass{}%
%\keywords{}%

%\date{}%
%\dedicatory{}%
%\commby{}%
% ----------------------------------------------------------------
%\begin{abstract}
%
%\end{abstract}
\maketitle

\tableofcontents

% ----------------------------------------------------------------

\section{Introduction}

In \cite{DiFo77b}, Diederich and Fornaess proved that for every bounded pseudoconvex domain $\Omega$ with $C^2$ boundary in a Stein manifold, there exists a defining function $\rho$ and an exponent $0<\eta<1$ such that $-(-\rho)^\eta$ is strictly plurisubharmonic on $\Omega$ (see also the simplified proof for $C^3$ domains in \cite{Ran81}).  Their starting point is Oka's Lemma, which guarantees that $-\log\delta$ is plurisubharmonic on $\Omega$ where $\delta(z)$ measures the distance from $z$ to the boundary of $\Omega$.  The proof proceeds by modifying $\delta$ with a global strictly plurisubharmonic exhaustion function ($\abs{z}^2$ in $\mathbb{C}^n$) and carefully studying the complex hessian of the resulting function.  In \cite{DiFo77a}, Diederich and Fornaess also showed that this result is sharp by showing that for every $0<\eta<1$ there exists a smooth bounded pseudoconvex domain in $\mathbb{C}^n$ for which no defining function $\rho$ exists such that $-(-\rho)^\eta$ is plurisubharmonic.

Ohsawa and Sibony generalized this result to $\mathbb{CP}^n$ in \cite{OhSi98}.  In this setting there is no global strictly plurisubharmonic function, but Takeuchi's Theorem \cite{Tak64} strengthens Oka's Lemma by showing that for pseudoconvex domains $\Omega\subset\mathbb{CP}^n$, $-\log\delta$ is strictly plurisubharmonic.  Takeuchi's result has been further generalized in \cite{Ele75}, \cite{Suz76}, and \cite{GrWu78}, in which it is shown that the strict plurisubharmonicity of $-\log\delta$ depends on the positive bisectional curvature of $\mathbb{CP}^n$, so the constant involved is independent of $\Omega$.  A new proof when the boundary is $C^2$ is provided in \cite{CaSh05}.  Ohsawa and Sibony were able to show that Takeuchi's Theorem implies that $-\delta^\eta$ is strictly plurisubharmonic for some $0<\eta<1$ when $\delta$ is $C^2$ (see also \cite{HaSh07} for related results).

In \cite{KeRo81}, Kerzman and Rosay showed that bounded pseudoconvex domains in $\mathbb{C}^n$ with $C^1$ boundaries always admit a bounded plurisubharmonic exhaustion function.  Once again, Oka's Lemma and the existence of a global strictly plurisubharmonic function provided the key tools.  Kerzman and Rosay introduced the idea of locally translating $-\log\delta$ in a direction transverse to the boundary to obtain a bounded plurisubharmonic function that reflected the boundary geometry.  By taking the supremum over a carefully chosen family of such functions they obtained a plurisubharmonic exhaustion function.  The global strictly plurisubharmonic function was used to patch these local functions together.  Demailly \cite{Dem87} refined this argument to obtain a bounded plurisubharmonic function for any bounded domain with Lipschitz boundary, and showed that the resulting function is comparable to $(\log\delta)^{-1}$.  A domain is said to have Lipschitz boundary if the boundary can locally be written as the graph of a Lipschitz function, which guarantees the existence of the local transverse direction required by Kerzman and Rosay.  By further refining the local construction, the author was able to show in \cite{Har08a} that the result of Diederich and Fornaess holds on all bounded domains with Lipschitz boundary in $\mathbb{C}^n$.

The main result of the present paper is the following:
\begin{thm}
\label{thm:main_theorem}
  Let $\Omega\subset\mathbb{CP}^n$ be a pseudoconvex domain with Lipschitz boundary.  Then there exists a Lipschitz defining function $\rho$ and an exponent $0<\eta<1$ such that
  \[
    i\ddbar(-(-\rho)^\eta)\geq C(-\rho)^\eta\omega
  \]
  in the sense of currents where $C>0$ and $\omega$ is the K\"ahler form for the Fubini-Study metric on $\mathbb{CP}^n$.
\end{thm}

\begin{rem}
  Since the worm domains of Diederich and Fornaess \cite{DiFo77a} can be imbedded in $\mathbb{CP}^n$, this result remains sharp.
\end{rem}

Ohsawa and Sibony's result \cite{OhSi98} shows that we can take $\rho=-\delta$ when the boundary is $C^2$.  Diederich and Fornaess have already shown that this is not always possible in $\mathbb{C}^n$ with the Euclidean metric \cite{DiFo77b}.  In contrast to the Ohsawa-Sibony result, we have the following for Lipschitz boundaries:
\begin{prop}
\label{prop:counterexample}
  For any $n\in\mathbb{N}$, there exists a Lipschitz domain $\Omega\subset\mathbb{CP}^n$ such that for any open neighborhood $U$ of $\partial\Omega$ and $0<\eta<1$, $-\delta^\eta$ fails to be plurisubharmonic on $U\cap\Omega$.
\end{prop}

As an application, consider Theorem 2 in \cite{CSW04} (see also Proposition 2.4 in \cite{CaSh07}).  By adapting techniques of Berndtsson and Charpentier \cite{BeCh00}, Cao, Shaw, and Wang show that the existence of a defining function with a positive Diederich-Fornaess exponent implies Sobolev regularity for the $\dbar$-Neumann family of operators.  Combining their proof with Theorem \ref{thm:main_theorem}, we have
\begin{thm}
  Let $\Omega\subset\mathbb{CP}^n$ be a pseudoconvex domain with Lipschitz boundary, and let $0<\eta<1$ be given by Theorem \ref{thm:main_theorem}.  Then the $\dbar$-Neumann operator $N_{p,q}$ exists on $L^2_{(p,q)}(\Omega)$ for all $0\leq p,q\leq n$ and the operators $N_{p,q}$, $\dbar N_{p,q}$, $\dbar^* N_{p,q}$, and the Bergman Projection $P_{p,q}$ are continuous in the $L^2$ Sobolev space $W^s(\Omega)$ whenever $0<s<\eta$.
\end{thm}
For further background on the $\dbar$-Neumann problem in complex manifolds, see \cite{FoKo72} or \cite{ChSh01}.  Details for $\mathbb{CP}^n$ are provided in \cite{CSW04}.

After introducing our notation and definitions in Section \ref{sec:notation}, we will compute derivatives and estimates for derivatives of our basic geometric objects in Section \ref{sec:geometric computations}.  As in \cite{KeRo81} and related papers, we will need a holomorphic map to translate $-\log\delta$, and because of the delicacy of our estimates, we will need a generator for the K\"ahler form that is compatible with this holomorphic map.  These will be constructed in Lemma \ref{lem:weight_and_map}.  Some of the results in this section are well-known, but we include them so that our paper will be self-contained (except for the proof of Takeuchi's Theorem).

In Section \ref{sec:lipschitz_boundary}, we develop our tools for dealing with Lipschitz boundaries.  Lemma \ref{lem:transverse_vector_field} will construct a map from a boundary point $p$ to another point $v(p)$ so that geodesics through $v(p)$ will be uniformly transverse to the boundary in a neighborhood of $p$.  Furthermore, in a sense that can be made precise by considering directional derivatives with respect to the tangent cone of $\partial\Omega$, $v(p)$ will be $C^1$ on $\partial\Omega$.  Critical properties of this map will be developed in the following lemmas.

Theorem \ref{thm:main_theorem} will be proven in Section \ref{sec:proof_main_theorem}.  We begin with Lemma \ref{lem:tools}, which unites the results of Sections \ref{sec:geometric computations} and \ref{sec:lipschitz_boundary} in the way that will be most helpful for our main proof.  As in \cite{KeRo81}, \cite{Dem87}, and \cite{Har08a}, the proof proceeds by locally translating $-\log\delta$ transversally to the boundary and patching these local functions together.  However, the lack of a global strictly plurisubharmonic function will complicate the patching argument.  The previously named papers use a finite cover of $\partial\Omega$ with a subordinate partition of unity.  On each set in the cover a local plurisubharmonic function is constructed and then these are patched together with the global strictly plurisubharmonic function controlling the error terms. However, $-\log\delta$ may not be able to control the resulting error terms.  Instead, we will construct a local defining function for every boundary point, depending on the boundary point in a $C^1$ way.  This will allow us to replace the traditional patching argument with an optimization argument that introduces controllable error terms.

We conclude in Section \ref{sec:counter_example} with the proof of Proposition \ref{prop:counterexample}.  This example demonstrates that at least some of the complexity in the proof of Theorem \ref{thm:main_theorem} is unavoidable, since we are not able to use a function as nice as the distance function.

\section{Notation and Definitions}
\label{sec:notation}

Let $\mathbb{CP}^n$ denote $\mathbb{C}^{n+1}\backslash\set{0}$ with the equivalence relation $w\sim z$ if $w=\lambda z$ for some $\lambda\in\mathbb{C}\backslash\set{0}$.  The equivalence class for the representative element $z\in\mathbb{C}^{n+1}\backslash\set{0}$ will be denoted $[z]\in\mathbb{CP}^{n}$, although we will abuse notation and identify the equivalence class $[z]$ with the representative element $z$ whenever this can be done without ambiguity.  As is customary, we will denote elements of $\mathbb{C}^{n+1}\backslash\set{0}$ representing elements of $\mathbb{CP}^n$ by $w=[w^1:\cdots:w^{n+1}]$.  We will use the dot product to represent the customary dot product on $\mathbb{C}^{n+1}$.  When this is intended to be hermitian, we will make this explicit (e.g., $a\cdot\bar{b}$).  We will use $\abs{\cdot}$ for the Euclidean length of a vector in $\mathbb{C}^{n+1}$.

We equip $\mathbb{CP}^n$ with the Fubini-Study metric given by the K\"ahler form $\omega=i\ddbar\log\abs{w}$.  Inner products of vectors in this metric will be denoted $\left<\cdot,\cdot\right>_\omega$, with $\abs{\cdot}_\omega$ denoting the length of a vector.  Under our normalization for the Fubini-Study metric, the distance between any two points in $\mathbb{CP}^n$ is given by
\[
  \dist(p,q)=\arccos\left(\frac{\abs{p\cdot \bar{q}}}{\abs{p}\abs{q}}\right)=\arcsin\abs{\frac{p}{\abs{p}}-\frac{(p\cdot\bar{q})q}{\abs{p}\abs{q}^2}}.
\]
If $0<\dist(p,q)<\frac{\pi}{2}$, the unique geodesic connecting $p$ to $q$, parameterized according to distance, can be represented by
\[
  \gamma_t(p,q)=\cos t\frac{p}{\abs{p}}+\sin t\cot(\dist(p,q))\left(\frac{\abs{p}q}{q\cdot\bar{p}}-\frac{p}{\abs{p}}\right),
\]
where $t\in\mathbb{R}$.  Indeed, $\frac{p}{\abs{p}}$ and $\cot(\dist(p,q))\left(\frac{\abs{p}q}{q\cdot\bar{p}}-\frac{p}{\abs{p}}\right)$ form an orthonormal set, so $\gamma_t(p,q)$ is the unique great circle satisfying $\gamma_0(p,q)=p$ and $[\gamma_{\dist(p,q)}(p,q)]=[q]$.

We will define tangent vectors as follows:
\[
  \Tan(\mathbb{CP}^n,p)=\set{u\in\mathbb{C}^{n+1}:u=0\text{ or }\frac{u}{\abs{u}}=\gamma_0'(p,q)\text{ for some }q\in\mathbb{CP}^n}.
\]
Note that we have chosen our representative element for $[\gamma_t]$ to satisfy $\abs{\gamma_t}=\abs{\gamma_t'}=1$ and $\gamma_t\cdot\overline{\gamma_t}'=0$, so we can compute the inner product of $u_1,u_2\in\Tan(\mathbb{CP}^n,p)$ with respect to the Fubini-Study metric by
\begin{equation}
\label{eq:geodesic_inner_products}
  \left<u_1,u_2\right>_\omega=\re\left(u_1\cdot\overline{u_2}\right).
\end{equation}

Since we will be working on Lipschitz domains, directional derivatives will be particularly important.  For $u\in\Tan(\mathbb{CP}^n,p)$ and a real-valued function $f$ that is Lipschitz in a neighborhood of $p$, let $q\in\mathbb{CP}^n$ satisfy $\frac{u}{\abs{u}}=\gamma_0'(p,q)$ and denote
\begin{align*}
  D_{p,u}^+f(p)&=\lim_{h\rightarrow 0^+}\frac{f(\gamma_{h\abs{u}}(p,q))-f(p)}{h},\\
  D_{p,u}^-f(p)&=\lim_{h\rightarrow 0^+}\frac{f(p)-f(\gamma_{-h\abs{u}}(p,q))}{h},
\end{align*}
when these limits exist.  If $D_{p,u}^+f(p)=D_{p,u}^-f(p)$, we will simply write $D_{p,u} f(p)$.  If $f$ is differentiable at $p$, we let $D_p f(p)$ denote the element of $\Tan(\mathbb{CP}^n,p)$ satisfying $D_{p,u}f(p)=\left<D_p f(p),u\right>_\omega$ for all $u\in\Tan(\mathbb{CP}^n,p)$.  When differentiating with respect to a single real variable $t$, we will simply write $D_t^+$ or $D_t^-$, since the $u$ is unnecessary.  For example, we have
\[
  D_{p,u}^\pm f(p)=\left.D_t^\pm f(\gamma_{t\abs{u}}(p,q))\right|_{t=0}.
\]

If $g$ is $\mathbb{CP}^n$-valued in a neighborhood of $p$, $D_{p,u}^\pm g(p)$ should be an element of $\Tan(\mathbb{CP}^n,g(p))$ when it exists, so when $\frac{u}{\abs{u}}=\gamma_0'(p,q)$ we have
\[
  D_{p,u}^\pm g(p)=\lim_{h\rightarrow 0^\pm}\frac{\dist(g(p),g(\gamma_{h\abs{u}}(p,q)))}{h}\gamma_0'(g(p),g(\gamma_{h\abs{u}}(p,q))).
\]
Equivalently,
\begin{equation}
\label{eq:vector_derivative}
  D_{p,u}^\pm g(p)=\lim_{h\rightarrow 0^\pm}h^{-1}\left(\frac{g(\gamma_{h\abs{u}}(p,q))}{\abs{g(\gamma_{h\abs{u}}(p,q))}}-\frac{g(\gamma_{h\abs{u}}(p,q))\cdot\overline{g(p)}}{\abs{g(p)}\abs{g(\gamma_{h\abs{u}}(p,q))}}\frac{g(p)}{\abs{g(p)}}\right).
\end{equation}
When $g$ is differentiable at $p$, $D_p g(p):\Tan(\mathbb{CP}^n,p)\rightarrow\Tan(\mathbb{CP}^n,g(p))$ is a linear map.

We will also need the distance to a geodesic, which we will denote $\dist(z,\gamma(p,q))$.  This will be explicitly computed in \eqref{eq:distance_to_geodesic} as part of the proof of Lemma \ref{lem:weight_and_map}.

Let $\Omega\subset\mathbb{CP}^n$ be a nonempty, open, connected set, and denote the distance function by
\[
  \delta(z)=\inf_{w\in\partial\Omega}\dist(z,w).
\]
Since $\Omega$ is assumed to have nonempty interior, we will always have $\sup_{\mathbb{CP}^n}\delta<\frac{\pi}{2}$.  Basic properties of the distance function in $\mathbb{R}^n$ were developed in Section 4 of \cite{Fed59} (see also \cite{Wei75}, \cite{KrPa81}, \cite{HeMc12}, and \cite{HaRa13}).

Following Federer, for $p\in\partial\Omega$, we define $\Tan(\partial\Omega,p)$ to be the set of all $u\in\Tan(\mathbb{CP}^n,p)$ such that either $u=0$ or for every $\frac{\pi}{2}>\eps>0$ there exists $q\in\partial\Omega$ such that $0<\dist(p,q)<\eps$ and $\abs{\frac{u}{\abs{u}}-\gamma_0'(p,q)}_\omega\leq\eps$.  For a function $f$ defined on a neighborhood of $p$ in $\partial\Omega$ and $u\in\Tan(\partial\Omega,p)\backslash\set{0}$, we define
\[
  D_{p,u}^{\partial\Omega}f(p)=\lim_{j\rightarrow\infty}\frac{f(p_j)-f(p)}{\dist(p_j,p)}
\]
if and only if this limit exists for every sequence $\{p_j\}\subset\partial\Omega$ satisfying $[p_j]\rightarrow[p]$ and $\gamma_0'(p,p_j)\rightarrow\frac{u}{\abs{u}}$, and the limit is independent of the choice of $\{p_j\}$.

In any metric space, it remains true that $\delta$ is a Lipschitz function with Lipschitz constant $1$.  While this only implies that the distance function is differentiable almost everywhere, we will see that directional derivatives of the distance function always exist in Lemma \ref{lem:directional_derivative}.  Building on a theme from \cite{Fed59}, the key to understanding differentiability of the distance function at $z$ lies in the set of boundary points minimizing the distance to $z$.  To that end, we define
\[
  \Pi_\Omega(z)=\set{p\in\partial\Omega:\delta(z)=\dist(z,p)}.
\]

Since the definition of Lipschitz boundary is essentially a local definition, we will now introduce some notation for working in local coordinate charts.  Define $U=\set{z\in\mathbb{CP}^n:z^{n+1}\neq 0}$.  As usual, $U\cong\mathbb{C}^n$ via the holomorphic map $\tilde{z}=\left(\frac{z^1}{z^{n+1}},\cdots,\frac{z^n}{z^{n+1}}\right)$, with inverse (modulo the equivalence relation) $[z]=[\tilde{z}^1:\cdots:\tilde{z}^n:1]$.  When we say that $\Omega\subset\mathbb{CP}^n$ is a Lipschitz domain, we mean that for every $p\in\partial\Omega$ there exists a rotation mapping $p$ to $[0:\cdots:0:1]$ such that in our local coordinate chart we can write
\begin{equation}
\label{eq:graph_definition}
  \Omega\cap B(p,R)=\set{z\in B(p,R):\im\tilde{z}^n<\varphi(\tilde{z}',\re\tilde{z}^n)}
\end{equation}
for some $R>0$, where $\varphi$ is a Lipschitz function and $\tilde{z}'=(\tilde{z}^1,\cdots,\tilde{z}^{n-1})$.

We say that a Lipschitz function $\rho:\mathbb{CP}^n\rightarrow\mathbb{R}$ is a Lipschitz defining function for $\Omega$ if $\Omega=\set{z\in\mathbb{CP}^n:\rho(z)<0}$ and $0<\inf\abs{\nabla\rho}<\sup\abs{\nabla\rho}<\infty$.  In contrast with the $C^k$ case, the existence of a Lipschitz defining function for $\Omega$ is not sufficient to guarantee that $\Omega$ is a Lipschitz domain, because of the failure of the implicit function theorem (see Section 1.2.1 in \cite{Gri85} for a counterexample).

\section{Geometric Computations}
\label{sec:geometric computations}

We begin this section by computing some derivatives of our fundamental geometric objects.  Our most important derivative is the following:
\begin{lem}
\label{lem:point_distance_derivative}
  Let $p,q\in\mathbb{CP}^n$ satisfy $0<\dist(p,q)<\frac{\pi}{2}$.  Then
  \begin{equation}
    \label{eq:point_distance_derivative}
    D_p\dist(p,q)=-\gamma_0'(p,q).
  \end{equation}
\end{lem}

\begin{proof}
  Pick $u\in\Tan(\mathbb{CP}^n,p)\backslash\set{0}$ and choose $z\in\mathbb{CP}^n$ satisfying $0<\dist(p,z)<\frac{\pi}{2}$ and $\frac{u}{\abs{u}}=\gamma_0'(p,z)$.  We compute
  \[
    \left.\frac{d}{dt}\dist(\gamma_{t\abs{u}}(p,z),q)\right|_{t=0}
    =-\csc(\dist(p,q))\sec(\dist(p,q))\re\left(\frac{(u\cdot\bar{q})(q\cdot\overline{p})}{\abs{q}^2\abs{p}}\right).
  \]
  Since $u\cdot\bar{p}=0$, we have
  \[
    \frac{(u\cdot\bar{q})(q\cdot\bar{p})}{\abs{q}^2\abs{p}}=\sin(\dist(p,q))\cos(\dist(p,q))u\cdot\overline{\gamma_0'(p,q)}.
  \]
  Using \eqref{eq:geodesic_inner_products}, \eqref{eq:point_distance_derivative} follows.
\end{proof}

Our most important $\mathbb{CP}^n$-valued maps are geodesics, so we will next compute the derivative in the second variable.
\begin{lem}
\label{lem:derivatives}
  Let $p,q\in\mathbb{CP}^n$ satisfy $0<\dist(p,q)<\frac{\pi}{2}$, and choose $u\in\Tan(\mathbb{CP}^n,q)$.  Then
  \begin{multline}
  \label{eq:geodesic_vector_derivative_q}
    D_{q,u}\gamma_0'(p,q)=\csc(\dist(p,q))\frac{\abs{q\cdot\bar{p}}}{q\cdot\bar{p}}\left(u-(u\cdot\overline{\gamma_0'(q,p)})\gamma_0'(q,p)\right)\\
    -i\sec(\dist(p,q))\csc(\dist(p,q))\im\left(u\cdot\overline{\gamma_0'(q,p)}\right)\gamma_0'(p,q)
  \end{multline}
  and
  \begin{multline}
  \label{eq:geodesic_derivative_q}
    D_{q,u}\gamma_t(p,q)=\sin t\csc(\dist(p,q))\frac{\abs{q\cdot\bar{p}}}{q\cdot\bar{p}}\left(u-(u\cdot\overline{\gamma_0'(q,p)})\gamma_0'(q,p)\right)\\
    -\sin 2t\csc(2\dist(p,q))i\im\left(u\cdot\overline{\gamma_0'(q,p)}\right)\gamma_t'(p,q).
  \end{multline}
\end{lem}

\begin{rem}
  Since \eqref{eq:geodesic_inner_products} implies that $\Tan(\mathbb{CP}^n,p)$ is isometric with $\mathbb{C}^n$ under the Euclidean metric, we interpret \eqref{eq:geodesic_vector_derivative_q} in the Euclidean sense.
\end{rem}

\begin{proof}
  Using \eqref{eq:point_distance_derivative}, we have
  \begin{multline*}
    D_{q,u}\gamma_0'(p,q)=\sec(\dist(p,q))\csc(\dist(p,q))\left<u,\gamma_0'(q,p)\right>_\omega\gamma_0'(p,q)\\
    +\cot(\dist(p,q))\left(\frac{\abs{p}\abs{q}u}{q\cdot\bar{p}}-\frac{\abs{p}\abs{q}(u\cdot\bar{p})q}{(q\cdot\bar{p})^2}\right).
  \end{multline*}
  We can substitute
  \begin{equation}
  \label{eq:q_in_terms_of_p}
    \frac{\abs{p}q}{q\cdot\bar{p}}=\frac{p}{\abs{p}}+\tan(\dist(p,q))\gamma_0'(p,q).
  \end{equation}
  Since $u\cdot\bar{q}=0$ we will also use \eqref{eq:q_in_terms_of_p} with $p$ and $q$ reversed to obtain $\frac{\abs{q}u\cdot\bar{p}}{q\cdot\bar{p}}=\tan(\dist(p,q))u\cdot\overline{\gamma_0'(q,p)}$.  Hence
  \begin{multline*}
    D_{q,u}\gamma_0'(p,q)=\sec(\dist(p,q))\csc(\dist(p,q))\left<u,\gamma_0'(q,p)\right>_\omega\gamma_0'(p,q)\\
    +\cot(\dist(p,q))\frac{\abs{p}\abs{q}u}{q\cdot\bar{p}}-\left(u\cdot\overline{\gamma_0'(q,p)}\right)\left(\frac{p}{\abs{p}}+\tan(\dist(p,q))\gamma_0'(p,q)\right).
  \end{multline*}
  To simplify notation we will define the orthogonal projection
  \[
    \hat{u}=\frac{\abs{q\cdot\bar{p}}}{q\cdot\bar{p}}\left(u-(u\cdot\overline{\gamma_0'(q,p)})\gamma_0'(q,p)\right),
  \]
  but this makes it helpful to note that from \eqref{eq:q_in_terms_of_p}, we have
  \begin{equation}
  \label{eq:q_switch_p}
    \gamma_0'(q,p)=\frac{q\cdot\bar{p}}{\abs{p}\abs{q}}\left(\frac{p}{\abs{p}}\tan(\dist(p,q))-\gamma_0'(p,q)\right),
  \end{equation}
  so $u=\frac{q\cdot\bar{p}}{\abs{q}\abs{p}}\left(\sec(\dist(p,q))\hat{u}+(u\cdot\overline{\gamma_0'(q,p)})\left(\frac{p}{\abs{p}}\tan(\dist(p,q))-\gamma_0'(p,q)\right)\right)$ and
  \begin{multline*}
    D_{q,u}\gamma_0'(p,q)=\sec(\dist(p,q))\csc(\dist(p,q))\left<u,\gamma_0'(q,p)\right>_\omega\gamma_0'(p,q)\\
    +\csc(\dist(p,q))\hat{u}-\left(u\cdot\overline{\gamma_0'(q,p)}\right)\left(\tan(\dist(p,q))+\cot(\dist(p,q))\right)\gamma_0'(p,q).
  \end{multline*}
  Trigonometric identities and \eqref{eq:geodesic_inner_products} can be used to obtain \eqref{eq:geodesic_vector_derivative_q}.

  Let $w\in\mathbb{CP}^n$ satisfy $0<\dist(p,w)<\frac{\pi}{2}$ and $\frac{u}{\abs{u}}=\gamma_0'(p,w)$.  We have
  \[
    \gamma_{h\abs{u}}(q,w)=\frac{q}{\abs{q}}+hu+O(h^2),
  \]
  so using \eqref{eq:geodesic_vector_derivative_q} we have
  \begin{multline*}
    \gamma_0'(p,\gamma_{h\abs{u}}(q,w))=\gamma_0'(p,q)+h\csc(\dist(p,q))\hat{u}\\
    -h\csc(\dist(p,q))\sec(\dist(p,q))i\im\left(u\cdot\overline{\gamma_0'(q,p)}\right)\gamma_0'(p,q)+O(h^2).
  \end{multline*}
  This give us
  \begin{multline*}
    \gamma_t(p,\gamma_{h\abs{u}}(q,w))=\gamma_t(p,q)+h\sin t\csc(\dist(p,q))\hat{u}\\
    -h\sin t\csc(\dist(p,q))\sec(\dist(p,q))i\im\left(u\cdot\overline{\gamma_0'(q,p)}\right)\gamma_0'(p,q)+O(h^2).
  \end{multline*}
  Since $\gamma_0'(p,q)=\sin t\gamma_t(p,q)+\cos t\gamma_t'(p,q)$, we have
  \begin{multline*}
    \gamma_t(p,\gamma_{h\abs{u}}(q,w))=h\sin t\csc(\dist(p,q))\hat{u}\\
    +\left(1-h\sin^2 t\sec(\dist(p,q))\csc(\dist(p,q))i\im\left(u\cdot\overline{\gamma_0'(q,p)}\right)\right)\gamma_t(p,q)\\
    -h\sin 2t\csc(2\dist(p,q))i\im\left(u\cdot\overline{\gamma_0'(q,p)}\right)\gamma_t'(p,q)+O(h^2).
  \end{multline*}
  Hence, \eqref{eq:geodesic_derivative_q} follows from \eqref{eq:vector_derivative}.
\end{proof}

\begin{cor}
\label{cor:lipschitz_geodesic}
  Let $p,q\in\mathbb{CP}^n$ satisfy $0<\dist(p,q)<\frac{\pi}{2}$.  Then
  \begin{equation}
  \label{eq:lipschitz_geodesic_gradient}
    \abs{D_q\gamma_0'(p,q)}\leq\sec(\dist(p,q))\csc(\dist(p,q)).
  \end{equation}
  For any $t\in\mathbb{R}$,
  \begin{equation}
  \label{eq:lipschitz_geodesic_q}
    \abs{D_q\gamma_t(p,q)}_\omega\leq\max\set{\abs{\sin 2t}\csc(2\dist(p,q)),\abs{\sin t}\csc(\dist(p,q))},
  \end{equation}
  and
  \begin{multline}
  \label{eq:lipschitz_geodesic_p}
    \abs{D_p\gamma_t(p,q)}_\omega\\
    \leq\max\set{\abs{\sin 2(t-\dist(p,q))}\csc(2\dist(p,q)),\abs{\sin(t-\dist(p,q))}\csc(\dist(p,q))}.
  \end{multline}
\end{cor}

\begin{proof}
  Suppose $u\in\Tan(\mathbb{CP}^n,q)$ satisfies $\abs{u}=1$.  From \eqref{eq:geodesic_vector_derivative_q}, we can compute
  \begin{multline*}
    \abs{D_{q,u}\gamma_0'(p,q)}^2=\csc^2(\dist(p,q))\left(1-\abs{u\cdot\overline{\gamma_0'(q,p)}}^2\right)\\
    +\sec^2(\dist(p,q))\csc^2(\dist(p,q))\left(\im\left(u\cdot\overline{\gamma_0'(q,p)}\right)\right)^2,
  \end{multline*}
  and this is bounded above by $\sec^2(\dist(p,q))\csc^2(\dist(p,q))$.

  From \eqref{eq:geodesic_derivative_q}, we can compute
  \begin{multline*}
    \abs{D_{q,u}\gamma_t(p,q)}^2=\sin^2 t\csc^2(\dist(p,q))\left(1-\abs{u\cdot\overline{\gamma_0'(q,p)}}^2\right)\\
    +\sin^2 2t\csc^2(2\dist(p,q))\left(\im\left(u\cdot\overline{\gamma_0'(q,p)}\right)\right)^2,
  \end{multline*}
  and this is bounded above by $\max\set{\sin^2 t\csc^2(\dist(p,q)),\sin^2 2t\csc^2(2\dist(p,q))}$.  Since $[\gamma_t(p,q)]=[\gamma_{\dist(p,q)-t}(q,p)]$, the Lipschitz constant in $p$ can easily be derived.

\end{proof}

Although $i\ddbar\log\abs{w}$ is a well-defined $(1,1)$-form, $\log\abs{w}$ is not a well-defined function on $\mathbb{CP}^n$.  When working locally, we will use a special class of functions to generate the Fubini-Study metric.  We will also find it helpful to have a family of holomorphic isometries that is compatible with our strictly plurisubharmonic function, as follows:
\begin{lem}
\label{lem:weight_and_map}
  Let $p,q\in\mathbb{CP}^n$ satisfy $0<\dist(p,q)<\frac{\pi}{2}$.  On the set
  \[
    \dom\mu_{p,q}=\set{z\in\mathbb{CP}^n:\dist\left(z,\frac{p}{\abs{p}}\pm\gamma_0'(p,q)\right)<\frac{\pi}{2}},
  \]
  there exists a real-valued function $\mu_{p,q}$ and a family of holomorphic isometries $\phi_t^{p,q}(z)$ preserving $\dom\mu_{p,q}$ such that
  \begin{enumerate}
    \item \label{item:w_domain} If $\dist(z,\gamma(p,q))<\frac{\pi}{4}$, then $z\in\dom\mu_{p,q}$.

    \item \label{item:w_kahler_generator} $\omega=i\ddbar\mu_{p,q}$ on $\dom\mu_{p,q}$.

    \item \label{item:w_special_case} $\mu_{p,q}(\gamma_t(p,q))=0$ for all $t$.

    \item \label{item:w_lower_bound} $\mu_{p,q}(z)\geq\log\sec\dist(z,\gamma(p,q))$ for all $z\in\dom\mu_{p,q}$.

    \item \label{item:w_invariance} If $\varphi$ is a holomorphic isometry of $\mathbb{CP}^n$, then $\mu_{p,q}(z)=\mu_{\varphi(p),\varphi(q)}(\varphi(z))$.  If $s_1,s_2\in\mathbb{R}$ satisfy $0<\abs{s_2-s_1}<\frac{\pi}{2}$, then $\mu_{p,q}(z)=\mu_{\gamma_{s_1}(p,q),\gamma_{s_2}(p,q)}(z)$.

    \item \label{item:w_second_derivative} For $z\in\dom\mu_{p,q}$, we have
        \begin{equation}
          \label{eq:w_second_derivative}
          \mu_{p,q}(z)=\frac{1}{2}(\dist(z,p))^2\left(1-\re\left(\left(\gamma_0'(p,z)\cdot\overline{\gamma_0'(p,q)}\right)^2\right)\right)
          +O((\dist(p,z))^3).
        \end{equation}

    \item \label{item:w_first_derivative} For $u_p\in\Tan(\mathbb{CP}^n,p)$, $u_q\in\Tan(\mathbb{CP}^n,q)$, and $z\in\dom\mu_{p,q}$, we have
        \begin{multline}
          \label{eq:w_first_derivative}
          (D_{p,u_p}+D_{q,u_q})\mu_{p,q}(z)\\=-\dist(z,p)
          \re\left(\gamma_0'(p,z)\cdot\overline{u_p}-\left(\gamma_0'(p,z)\cdot\overline{\gamma_0'(p,q)}\right)\left(u_p\cdot\overline{\gamma_0'(p,q)}\right)\right)\\
          +O((\dist(z,p))^2).
        \end{multline}

    \item \label{item:m_invariance} If $\varphi$ is a holomorphic isometry of $\mathbb{CP}^n$, then $\varphi(\phi_t^{p,q}(z))=\phi_t^{\varphi(p),\varphi(q)}(\varphi(z))$.  If $s_1,s_2\in\mathbb{R}$ satisfy $0<s_2-s_1<\frac{\pi}{2}$, then $\phi_t^{p,q}(z)=\phi_t^{\gamma_{s_1}(p,q),\gamma_{s_2}(p,q)}(z)$

    \item \label{item:m_special_cases} $\phi_0^{p,q}(z)=z$ for all $z\in\dom\mu_{p,q}$ and $\phi_t^{p,q}(\gamma_s(p,q))=\gamma_{t+s}(p,q)$ for all $s,t\in\mathbb{R}$.

    \item \label{item:m_level_curves} $\mu_{p,q}(\phi_t^{p,q}(z))=\mu_{p,q}(z)$ for all $z\in\dom\mu_{p,q}$ and $t\in\mathbb{R}$.

    \item \label{item:m_Lipschitz_in_p} If $p,q\in\mathbb{CP}^n$ satisfy $\dist(p,q)=\frac{\pi}{4}$ and $u_p\in\Tan(\mathbb{CP}^n,p)$ and $u_q\in\Tan(\mathbb{CP}^n,q)$ satisfy $(D_{p,u_p}+D_{q,u_q})\dist(p,q)=0$, then for all $0\leq t\leq\frac{\pi}{4}$ we have
        \begin{equation}
        \label{eq:m_Lipschitz_in_p}
          \left.\abs{(D_{p,u_p}+D_{q,u_q})\phi_t^{p,q}(z)}_\omega\right|_{z=p}\leq\sqrt{6}\sin t(\abs{u_p}_\omega+\abs{u_q}_\omega)
        \end{equation}

  \end{enumerate}
\end{lem}

\begin{rem}
  In the Euclidean case, we could take $\mu_{p,q}(z)=\frac{1}{2}\abs{z}^2-\frac{1}{2}\re\left(z\cdot\frac{q-p}{\abs{q-p}}\right)^2$ and $\phi_t^{p,q}(z)=z+t\frac{q-p}{\abs{q-p}}$ and obtain the same result.  Note that all trajectories of $\phi_t^{p,q}(z)$ are geodesic in the Euclidean case, but this is impossible in the projective case.
\end{rem}

\begin{proof}
  Set
  \begin{align*}
    \alpha&=\frac{1}{\sqrt{2}}\left(\frac{p}{\abs{p}}+i\gamma_0'(p,q)\right),\\ \beta&=\frac{1}{\sqrt{2}}\left(\frac{p}{\abs{p}}-i\gamma_0'(p,q)\right).
  \end{align*}
  Observe that $\alpha\cdot\bar\beta=0$ and $\abs{\alpha}^2=\abs{\beta}^2=1$.  In this notation, we have
  \begin{equation}
  \label{eq:gamma_using_alpha_beta}
    \gamma_t(p,q)=\frac{1}{\sqrt{2}}(\alpha e^{-it}+\beta e^{it}).
  \end{equation}
  Furthermore,
  \[
    \cos^2\dist(z,\gamma_t(p,q))=\frac{1}{2}\abs{\alpha\cdot\frac{\bar{z}}{\abs{z}}}^2+\re\left(\left(\alpha\cdot\frac{\bar z}{\abs{z}}\right)\left(\frac{z}{\abs{z}}\cdot\bar\beta\right)e^{-2it}\right)+\frac{1}{2}\abs{\beta\cdot\frac{\bar{z}}{\abs{z}}}^2,
  \]
  This is maximized when we choose $t$ satisfying $\re\left(\left(\alpha\cdot\frac{\bar z}{\abs{z}}\right)\left(\frac{z}{\abs{z}}\cdot\bar\beta\right)e^{-2it}\right)=\abs{\alpha\cdot\frac{\bar{z}}{\abs{z}}}\abs{\beta\cdot\frac{\bar{z}}{\abs{z}}}$, so we have
  \begin{equation}
  \label{eq:distance_to_geodesic}
    \dist(z,\gamma(p,q))=\arccos\left(\frac{1}{\sqrt{2}}\left(\abs{\alpha\cdot\frac{\bar{z}}{\abs{z}}}+\abs{\beta\cdot\frac{\bar{z}}{\abs{z}}}\right)\right).
  \end{equation}
  If $\dist(z,\alpha)=\frac{\pi}{2}$, then $z\cdot\bar\alpha=0$, so
  \[
    \dist(z,\gamma(p,q))=\arccos\left(\frac{1}{\sqrt{2}}\abs{\beta\cdot\frac{\bar{z}}{\abs{z}}}\right)\geq\frac{\pi}{4}.
  \]
  Similarly, if $\dist(z,\alpha)=\frac{\pi}{2}$, then $\dist(z,\gamma(p,q))\geq\frac{\pi}{4}$, so \eqref{item:w_domain} follows.

  For $z\in\dom\mu_{p,q}$, define
  \[
    \mu_{p,q}(z)=-\frac{1}{2}\log(2\cos\dist(z,\alpha)\cos\dist(z,\beta)).
  \]
  Observe that we can write $\mu_{p,q}(z)=\log\abs{z}-\frac{1}{2}\log(\sqrt{2}\abs{z\cdot\bar\alpha})-\frac{1}{2}\log(\sqrt{2}\abs{z\cdot\bar\beta})$.  Since $z\cdot\bar\alpha$ and $z\cdot\bar\beta$ are both holomorphic functions, $\log\abs{z\cdot\bar\alpha}$ and $\log\abs{z\cdot\bar\beta}$ are pluriharmonic, so we have \eqref{item:w_kahler_generator}.  Substituting \eqref{eq:gamma_using_alpha_beta} into $\mu_{p,q}$ will give us \eqref{item:w_special_case}.  Since \eqref{eq:distance_to_geodesic} coupled with the classical inequality $\frac{a+b}{2}\geq\sqrt{ab}$ implies
  \[
    \cos\dist(z,\gamma(p,q))\geq\sqrt{2\abs{\alpha\cdot\frac{\bar{z}}{\abs{z}}}\abs{\beta\cdot\frac{\bar{z}}{\abs{z}}}}=e^{-\mu_{p,q}(z)},
  \]
  we obtain \eqref{item:w_lower_bound}.

  Every holomorphic isometry of $\mathbb{CP}^n$ can be represented by a unitary map $U$ on $\mathbb{C}^{n+1}$.  The first part of \eqref{item:w_invariance} follows immediately since $U z\cdot\overline{U \alpha}=z\cdot\bar\alpha$, $U z\cdot\overline{U \beta}=z\cdot\bar\beta$, $U\frac{p}{\abs{p}}=\frac{U p}{\abs{U p}}$, and $U\gamma_0'(p,q)=\gamma_0'(U p,U q)$.  Now, suppose $s_1,s_2\in\mathbb{R}$ satisfy $0<\abs{s_2-s_1}<\frac{\pi}{2}$, and set $\hat{p}=\gamma_{s_1}(p,q)$ and $\hat{q}=\gamma_{s_2}(p,q)$.  Then $\gamma_0'(\hat{p},\hat{q})=\frac{\tan(s_2-s_1)}{\abs{\tan(s_2-s_1)}}\gamma_{s_1}'(p,q)$.  When $\tan(s_2-s_1)>0$, we have $\hat\alpha=e^{-is_1}\alpha$ and $\hat\beta=e^{is_1}\beta$, while if $\tan(s_2-s_1)<0$, we have $\hat\alpha=e^{is_1}\beta$ and $\hat\beta=e^{-is_1}\alpha$.  In either case, $\abs{z\cdot\hat\alpha}\abs{z\cdot\hat\beta}=\abs{z\cdot\alpha}\abs{z\cdot\beta}$, so the proof of \eqref{item:w_invariance} is complete.

  Let $t=\dist(z,p)$ and assume $[z]$ is represented by $\gamma_t(p,z)$.  Then
  \begin{align*}
    z\cdot\bar\alpha=\frac{1}{\sqrt{2}}\left(1-t^2/2-it\gamma_0'(p,z)\cdot\overline{\gamma_0'(p,q)}\right)+O(t^3),\\ z\cdot\bar\beta=\frac{1}{\sqrt{2}}\left(1-t^2/2+it\gamma_0'(p,z)\cdot\overline{\gamma_0'(p,q)}\right)+O(t^3),
  \end{align*}
  so
  \begin{align*}
    2\abs{z\cdot\bar\alpha}^2=1-t^2+2t\im(\gamma_0'(p,z)\cdot\overline{\gamma_0'(p,q)})+t^2\abs{\gamma_0'(p,z)\cdot\overline{\gamma_0'(p,q)}}^2+O(t^3),\\
    2\abs{z\cdot\bar\beta}^2=1-t^2-2t\im(\gamma_0'(p,z)\cdot\overline{\gamma_0'(p,q)})+t^2\abs{\gamma_0'(p,z)\cdot\overline{\gamma_0'(p,q)}}^2+O(t^3),
  \end{align*}
  and hence
  \begin{multline*}
    4\abs{z\cdot\bar\alpha}^2\abs{z\cdot\bar\beta}^2=1-2t^2+2t^2\abs{\gamma_0'(p,z)\cdot\overline{\gamma_0'(p,q)}}^2\\
    -4t^2\left(\im(\gamma_0'(p,z)\cdot\overline{\gamma_0'(p,q)})\right)^2+O(t^3).
  \end{multline*}
  Using the linear approximation $\log x=-1+x+O(x^2)$, we have
  \[
    \mu_{p,q}(z)=\frac{1}{2}t^2-\frac{1}{2}t^2\re\left(\left(\gamma_0'(p,z)\cdot\overline{\gamma_0'(p,q)}\right)^2\right)
    +O(t^3),
  \]
  so \eqref{eq:w_second_derivative} follows.

  For $u_p\in\Tan(\mathbb{CP}^n,p)$ and $u_q\in\Tan(\mathbb{CP}^n,q)$, we define
  \[
    p_s=\begin{cases}\cos(\abs{u_p}s)\frac{p}{\abs{p}}+\sin(\abs{u_p}s)\frac{u_p}{\abs{u_p}}&u_p\neq 0\\\frac{p}{\abs{p}}&u_p=0\end{cases},
  \]
  and
  \[
    q_s=\begin{cases}\cos(\abs{u_q}s)\frac{q}{\abs{q}}+\sin(\abs{u_q}s)\frac{u_q}{\abs{u_q}}&u_q\neq 0\\\frac{q}{\abs{q}}&u_q=0\end{cases}.
  \]
  Now,
  \[
    \left.\frac{d}{ds}\mu_{p_s,q_s}(z)\right|_{s=0}=-\frac{1}{2}\re\left(\frac{z\cdot\bar\alpha_0'}{z\cdot\bar\alpha_0}\right)-\frac{1}{2}\re\left(\frac{z\cdot\bar\beta_0'}{z\cdot\bar\beta_0}\right)
  \]
  Once again we set $t=\dist(z,p)$ and assume that $[z]$ is represented by $z=\gamma_t(p,z)$.  Then
  \begin{align*}
    z\cdot\bar\alpha_0&=\frac{1}{\sqrt{2}}\left(1-it\gamma_0'(p,z)\cdot\overline{\gamma_0'(p,q)}\right)+O(t^2),\\
    z\cdot\bar\beta_0&=\frac{1}{\sqrt{2}}\left(1+it\gamma_0'(p,z)\cdot\overline{\gamma_0'(p,q)}\right)+O(t^2),
  \end{align*}
  so
  \begin{align*}
    \frac{z\cdot\bar\alpha_0'}{z\cdot\bar\alpha_0}&=\sqrt{2}\left(\frac{p}{\abs{p}}\cdot\bar\alpha_0'+t\gamma_0'(p,z)\cdot\bar\alpha_0'+it\left(\frac{p}{\abs{p}}\cdot\bar\alpha_0'\right)\left(\gamma_0'(p,z)\cdot\overline{\gamma_0'(p,q)}\right)\right)+O(t^2),\\
    \frac{z\cdot\bar\beta_0'}{z\cdot\bar\beta_0}&=\sqrt{2}\left(\frac{p}{\abs{p}}\cdot\bar\beta_0'+t\gamma_0'(p,z)\cdot\bar\beta_0'-it\left(\frac{p}{\abs{p}}\cdot\bar\beta_0'\right)\left(\gamma_0'(p,z)\cdot\overline{\gamma_0'(p,q)}\right)\right)+O(t^2).
  \end{align*}
  Since $p_s\cdot\alpha_s=p_s\cdot\beta_s=\frac{1}{\sqrt{2}}$, we can differentiate and obtain
  \begin{align}
  \label{eq:p_cdot_alpha_derivative}
    \frac{p}{\abs{p}}\cdot\bar\alpha_0'&=-u_p\cdot\bar\alpha_0=\frac{i}{\sqrt{2}}u_p\cdot\overline{\gamma_0'(p,q)},\\
  \label{eq:p_cdot_beta_derivative}
    \frac{p}{\abs{p}}\cdot\bar\beta_0'&=-u_p\cdot\bar\beta_0=-\frac{i}{\sqrt{2}}u_p\cdot\overline{\gamma_0'(p,q)}.
  \end{align}
  We know $\alpha_0'+\beta_0'=\sqrt{2}u_p$, so
  \begin{multline*}
    \left.\frac{d}{ds}\mu_{p_s,q_s}(z)\right|_{s=0}\\
    =-t\re\left(\gamma_0'(p,z)\cdot\bar{u}_p-\left(u_p\cdot\overline{\gamma_0'(p,q)}\right)\left(\gamma_0'(p,z)\cdot\overline{\gamma_0'(p,q)}\right)\right)+O(t^2),
  \end{multline*}
  and \eqref{eq:w_first_derivative} follows.

  Set
  \[
    \phi_t^{p,q}(z)=z-(z\cdot\bar\alpha)\alpha-(z\cdot\bar\beta)\beta+e^{-it}(z\cdot\bar\alpha)\alpha+e^{it}(z\cdot\bar\beta)\beta.
  \]
  The same computations used to show \eqref{item:w_invariance} will also imply \eqref{item:m_invariance}, if we note that our assumptions now restrict us to the case where $\tan(s_2-s_1)>0$.  By using \eqref{eq:gamma_using_alpha_beta}, we can easily check \eqref{item:m_special_cases}.  Since $\abs{\phi_t^{p,q}(z)}=\abs{z}$ and $\phi_t^{p,q}(z)$ is linear in $z$, $\phi_t^{p,q}(z)$ is a unitary map on $\mathbb{C}^{n+1}$, and hence a holomorphic isometry on $\mathbb{CP}^n$.  Since $\phi_t^{p,q}(z)\cdot\bar\alpha=e^{-it}z\cdot\bar\alpha$, and $\phi_t^{p,q}(z)\cdot\bar\beta=e^{it}z\cdot\bar\beta$, we must have $\dist(\phi_t^{p,q}(z),\alpha)=\dist(z,\alpha)$ and $\dist(\phi_t^{p,q}(z),\beta)=\dist(z,\beta)$.  It follows immediately that $\phi_t^{p,q}$ preserves $\dom\mu_{p,q}$ and \eqref{item:m_level_curves} holds.

  Now, we assume $\dist(p,q)=\frac{\pi}{4}$ and $(D_{p,u_p}+D_{q,u_q})\dist(p,q)=0$.  By \eqref{eq:point_distance_derivative}, we have $\left<u_p,\gamma_0'(p,q)\right>_\omega=-\left<u_q,\gamma_0'(q,p)\right>_\omega$.  Using \eqref{eq:geodesic_inner_products}, we will find it helpful to introduce the notation $u_p\cdot\overline{\gamma_0'(p,q)}=x+iy_p$ and $u_q\cdot\overline{\gamma_0'(q,p)}=-x+iy_q$ for some $x,y_p,y_q\in\mathbb{R}$.  We will also use the orthogonal projections $\hat u_p=u_p-(x+iy_p)\gamma_0'(p,q)$ and $\hat u_q=\frac{\abs{q\cdot\bar{p}}}{q\cdot\bar{p}}\left(u_q-(-x+iy_q)\gamma_0'(q,p)\right)$. Using \eqref{eq:q_switch_p}, we have $\gamma_0'(q,p)=\frac{q\cdot\bar{p}}{\abs{q}\abs{p}}\left(\frac{p}{\abs{p}}-\gamma_0'(p,q)\right)$.  Let $p_s$ and $q_s$ be as in the proof of \eqref{item:w_first_derivative}.  We can use \eqref{eq:p_cdot_alpha_derivative} and \eqref{eq:p_cdot_beta_derivative} to show
  \[
    \left.\frac{d}{ds}\phi_t^{p_s,q_s}(p)\right|_{s=0}=\frac{e^{-it}-1}{\sqrt{2}}\left((ix-y_p)\alpha_0+\alpha_0'\right)
    +\frac{e^{it}-1}{\sqrt{2}}\left((y_p-ix)\beta_0+\beta_0'\right).
  \]
  To compute $\alpha_0'$ and $\beta_0'$, we check
  \[
    \left.\frac{d}{ds}\gamma_0'(p_s,q_s)\right|_{s=0}=\frac{\abs{p}\abs{q}u_q}{q\cdot\bar{p}}-\frac{\abs{p}(\abs{q}u_q\cdot\bar{p}+\abs{p}q\cdot\bar u_p)q}{(q\cdot\bar{p})^2}-u_p.
  \]
  Since $p\cdot\bar u_p=0$, $\frac{\abs{p}q\cdot\bar u_p}{q\cdot\bar{p}}=x-iy_p$, and similarly $\frac{\abs{q}u_q\cdot\bar{p}}{q\cdot\bar{p}}=-x+iy_q$.  Using \eqref{eq:q_in_terms_of_p}, we have
  \[
    \left.\frac{d}{ds}\gamma_0'(p_s,q_s)\right|_{s=0}=\frac{\abs{p}\abs{q}u_q}{q\cdot\bar{p}}+i(y_p-y_q)\left(\frac{p}{\abs{p}}+\gamma_0'(p,q)\right)-u_p.
  \]
  Introducing our projections $\hat u_p$ and $\hat u_q$ yields
  \[
    \left.\frac{d}{ds}\gamma_0'(p_s,q_s)\right|_{s=0}=\sqrt{2}\hat u_q-\hat u_p+(-x+iy_p)\frac{p}{\abs{p}}-2iy_q\gamma_0'(p,q).
  \]
  Thus, we have the orthogonal decompositions
  \begin{align*}
    \alpha_0'&=\frac{1}{\sqrt{2}}\left((1-i)\hat u_p+i\sqrt{2}\hat u_q-(ix+y_p)\frac{p}{\abs{p}}+(x+iy_p+2y_q)\gamma_0'(p,q)\right),\\
    \beta_0'&=\frac{1}{\sqrt{2}}\left((1+i)\hat u_p-i\sqrt{2}\hat u_q+(ix+y_p)\frac{p}{\abs{p}}+(x+iy_p-2y_q)\gamma_0'(p,q)\right),
  \end{align*}
  so
  \begin{align*}
    (ix-y_p)\alpha_0+\alpha_0'&=\frac{1}{\sqrt{2}}\left((1-i)\hat u_p+i\sqrt{2}\hat u_q-2y_p\frac{p}{\abs{p}}+2y_q\gamma_0'(p,q)\right),\\
    (y_p-ix)\beta_0+\beta_0'&=\frac{1}{\sqrt{2}}\left((1+i)\hat u_p-i\sqrt{2}\hat u_q+2y_p\frac{p}{\abs{p}}-2y_q\gamma_0'(p,q)\right).
  \end{align*}
  Substituting, we find that
  \[
    \left.\frac{d}{ds}\phi_t^{p_s,q_s}(p)\right|_{s=0}=(\cos t-1)\hat u_p
    -\sin t\left(\hat u_p-\sqrt{2}\hat u_q-2iy_p\frac{p}{\abs{p}}+2iy_q\gamma_0'(p,q)\right).
  \]
  Since $\frac{p}{\abs{p}}=\cos t\gamma_t(p,q)-\sin t\gamma_t'(p,q)$ and $\gamma_0'(p,q)=\sin t\gamma_t(p,q)+\cos t\gamma_t'(p,q)$, we have
  \begin{multline*}
    \left.\frac{d}{ds}\phi_t^{p_s,q_s}(p)\right|_{s=0}=(\cos t-1)\hat u_p
    -\sin t\left(\hat u_p-\sqrt{2}\hat u_q\right)\\
    -2i\sin t\left((y_q\cos t+y_p\sin t)\gamma_t'(p,q)+(y_q\sin t-y_p\cos t)\gamma_t(p,q)\right).
  \end{multline*}
  We know $\phi_t^{p,q}(p)=\gamma_t(p,q)$ from \eqref{item:m_special_cases}, so \eqref{eq:vector_derivative} implies
  \begin{multline*}
    \left.(D_{p,u_p}+D_{q,u_q})\phi_t^{p,q}(z)\right|_{z=p}=(\cos t-1)\hat u_p
    -\sin t\left(\hat u_p-\sqrt{2}\hat u_q\right)\\
    -2i\sin t(y_q\cos t+y_p\sin t)\gamma_t'(p,q).
  \end{multline*}
  We can therefore estimate:
  \begin{multline*}
    \left.\abs{(D_{p,u_p}+D_{q,u_q})\phi_t^{p,q}(z)}^2_\omega\right|_{z=p}\leq\left(\abs{\cos t-1-\sin t}\abs{\hat u_p}_\omega+\sqrt{2}\sin t\abs{\hat u_q}_\omega\right)^2\\
    +4\sin^2 t(y_q\cos t+y_p\sin t)^2.
  \end{multline*}
  When $0<t\leq\frac{\pi}{4}$ we have $0\leq\frac{-\cos t+1+\sin t}{\sin t}\leq\sqrt{2}$, so
  \[
    \abs{\cos t-1-\sin t}\abs{\hat u_p}_\omega+\sqrt{2}\sin t\abs{\hat u_q}_\omega\leq\sqrt{2}\sin t(\abs{u_p}_\omega+\abs{u_q}_\omega)
  \]
  and since $\abs{y_q\cos t+y_p\sin t}\leq\abs{u_p}_\omega+\abs{u_q}_\omega$, \eqref{eq:m_Lipschitz_in_p} follows.

\end{proof}

Our final object of study in this section will be the distance function for the boundary of a domain.  As in \cite{Fed59}, differentiability of $\delta$ at a point $z$ will depend on the set of boundary points minimizing the distance to $z$.  We will see that the directional derivatives of $\delta$ always exist off of $\partial\Omega$, and they are completely determined by the set $\Pi_\Omega(z)$.

\begin{lem}
\label{lem:directional_derivative}
  Let $\Omega\subset\mathbb{CP}^n$ be a domain.  Let $z\notin\partial\Omega$ and $u\in\Tan(\mathbb{CP}^n,z)$.  Then we have
  \begin{equation}
  \label{eq:directional_derivative}
    D_{z,u}^+\delta(z)=-\sup_{p\in\Pi_\Omega(z)}\left<u,\gamma'_0(z,p)\right>_\omega
  \end{equation}
\end{lem}

\begin{rem}
  As in \cite{Fed59}, we may use this to show that $\delta$ is differentiable at $z$ if and only if $\abs{\Pi_\Omega(z)}=1$, in which case $D_z\delta(z)=-\gamma_0'(z,p)$ for $\set{p}=\Pi_\Omega(z)$.
\end{rem}

\begin{proof}
  When $u\neq 0$, pick $w\in B(z,\delta(z))\backslash\set{z}$ such that $\gamma_0'(z,w)=\frac{u}{\abs{u}}$.  For $p\in\Pi_\Omega(z)$ we must have $\delta(w)\leq\dist(w,p)$, so $\frac{\delta(w)-\delta(z)}{\dist(z,w)}\leq\frac{\dist(w,p)-\dist(z,p)}{\dist(z,w)}$.  We can apply \eqref{eq:point_distance_derivative} to obtain
  \begin{equation}
  \label{eq:difference_quotient_estimate}
    \frac{\delta(w)-\delta(z)}{\dist(z,w)}
    \leq-\left<u,\gamma'_0(z,p)\right>_\omega+O(\dist(z,w)).
  \end{equation}
  Homogeneity of the Fubini-Study metric implies that the constant in the error term can be chosen independently of $p$, $z$, and $w$, provided that we preserve a uniform lower bound on $\dist(z,p)$ and uniform upper bounds on $\dist(z,w)$ and $\dist(z,p)$.  Since this holds for all $p\in\Pi_\Omega(z)$, we may take the infimum over all such $p$.  If we replace $w$ with $\gamma_h(z,w)$ and take the limit in $h$, we have
  \begin{equation}
  \label{eq:directional_derivative_first_half}
    \limsup_{h\rightarrow 0^+}\frac{\delta(\gamma_h(z,w))-\delta(z)}{h}
    \leq -\sup_{p\in\Pi_\Omega(z)}\left<u,\gamma'_0(z,p)\right>_\omega.
  \end{equation}

  Before tackling the limit infimum, we claim that
  \begin{multline}
  \label{eq:liminf_of_supremum}
    \liminf_{h\rightarrow 0^+}\sup_{p\in\Pi_\Omega(\gamma_h(z,w))}\left<\gamma'_0(\gamma_h(z,w),z),\gamma'_0(\gamma_h(z,w),p)\right>_\omega\\
    \geq-\sup_{p\in\Pi_\Omega(z)}\left<u,\gamma'_0(z,p)\right>_\omega.
  \end{multline}
  To show this, we first use \eqref{eq:q_switch_p} to check
  \[
    \gamma'_0(\gamma_h(z,w),z)=\cos h\left(\frac{z}{\abs{z}}\tan h-\gamma_0'(z,w)\right)\rightarrow-u.
  \]
  Now, let $\set{h_j}$ be any positive sequence decreasing to zero.  For every $h_j$, choose $p_j\in\Pi_\Omega(\gamma_{h_j}(z,w))$ satisfying
  \begin{multline*}
    \left<\gamma'_0(\gamma_{h_j}(z,w),z),\gamma'_0(\gamma_{h_j}(z,w),p_j)\right>_\omega\\
    =\sup_{p\in\Pi_\Omega(\gamma_{h_j}(z,w))}\left<\gamma'_0(\gamma_{h_j}(z,w),z),\gamma'_0(\gamma_{h_j}(z,w),p)\right>_\omega.
  \end{multline*}
  Every limit point $p_\infty$ of $\set{p_j}$ must lie in $\Pi_\Omega(z)$, so we have
  \[
    \liminf_{j\rightarrow\infty}\left<\gamma'_0(\gamma_{h_j}(z,w),z),\gamma'_0(\gamma_{h_j}(z,w),p_j)\right>_\omega\geq\inf_{p\in\Pi_\Omega(z)}\left<-u,\gamma'_0(z,p)\right>_\omega.
  \]
  Since this holds for every such sequence, \eqref{eq:liminf_of_supremum} must follow.

  Now, we apply \eqref{eq:difference_quotient_estimate} with $\gamma_h(z,w)$ substituted for $z$ and $z$ substituted for $w$.  Taking limits with \eqref{eq:liminf_of_supremum} gives us
  \begin{equation}
  \label{eq:directional_derivative_second_half}
    \limsup_{h\rightarrow 0^+}\frac{\delta(z)-\delta(\gamma_h(z,w))}{h}\\
    \leq\sup_{p\in\Pi_\Omega(z)}\left<u,\gamma'_0(z,p)\right>_\omega.
  \end{equation}
  Combining \eqref{eq:directional_derivative_first_half} and \eqref{eq:directional_derivative_second_half} proves that the limit exists, and hence \eqref{eq:directional_derivative} follows.

\end{proof}

\section{Lipschitz Boundaries}
\label{sec:lipschitz_boundary}

Although the Lipschitz property is locally helpful, we will need a global object in our main construction.  Ultimately, we will want to work with the holomorphic isometry $\phi_t^{p,q}(z)$ given by Lemma \ref{lem:weight_and_map}.  However, since the trajectories of $\phi_t^{p,q}(z)$ are not generally geodesic, the computations in this section will be simplified if we instead work with the family of geodesics $\gamma_t(z,q)$ through a single point $q$.  Eventually, we will need to make use of the fact that $\gamma_t(z,q)\approx\phi_t^{p,q}(z)$ for $z$ sufficiently close to $p$ (see Lemma \ref{lem:tools}), so that we can combine the results in this section with results from the previous section.  In order to parameterize this family of maps, we introduce a map $v(p)$ that we may substitute for $q$.

\begin{lem}
\label{lem:transverse_vector_field}
  Let $\Omega\subset\mathbb{CP}^n$ be a Lipschitz domain.  There exists a Lipschitz map $v:\partial\Omega\rightarrow\mathbb{CP}^n$ and constants $R_0>0$ and $A_0>0$ such that for every $p\in\partial\Omega$, $z\in B(p,R_0)\cap\Omega$, and $q\in\Pi_\Omega(z)$, we have $\dist(p,v(p))=\frac{\pi}{4}$ and
  \begin{equation}
  \label{eq:vector_inner_product}
    \left<\gamma'_0(z,v(p)),\gamma'_0(z,q)\right>_\omega<-A_0.
  \end{equation}
  Furthermore, $v$ is $C^1$ in the sense that there exists a continuous family of linear maps $D_p v(p):\Tan(\mathbb{CP}^n,p)\rightarrow\Tan(\mathbb{CP}^n,v(p))$ such that for all $u\in\Tan(\partial\Omega,p)$ the derivative $D_{p,u}^{\partial\Omega}v(p)$ exists and $D_{p,u}^{\partial\Omega}v(p)=(D_p v(p))(u)$.
\end{lem}

\begin{proof}
  For each $p\in\partial\Omega$, choose a rotation so that $\Omega$ can be expressed near $p$ by \eqref{eq:graph_definition}.  Let $M$ be the Lipschitz constant of $\varphi$, and choose any $0<A_p<\frac{1}{\sqrt{M^2+1}}$.  It will simplify our computations to observe that the Fubini-Study metric is approximately the Euclidean metric near the origin of our local coordinate patch.  More precisely,
  \begin{multline*}
    \sin^2\dist(z,w)=\frac{\abs{\tilde{z}-\tilde{w}}^2+\abs{\tilde{z}}^2\abs{\tilde{w}}^2-\abs{z\cdot\bar{\tilde{w}}}^2}{(\abs{\tilde{z}}^2+1)(\abs{\tilde{w}}^2+1)}\\
    =\abs{\tilde z-\tilde w}^2+O\left(\abs{\tilde{z}-\tilde{w}}^2\max\set{\abs{\tilde{w}}^2,\abs{\tilde{z}}^2}\right),
  \end{multline*}
  so
  \begin{equation}
  \label{eq:approximate_distance}
    \dist(z,w)=\abs{\tilde z-\tilde w}+O\left(\abs{\tilde{z}-\tilde{w}}\max\set{\abs{\tilde{w}}^2,\abs{\tilde{z}}^2}\right).
  \end{equation}
  Furthermore,
  \[
    \gamma_0'(z,w)=\frac{\left((\abs{\tilde{z}}^2+1)\tilde w-(\tilde{w}\cdot\bar{\tilde{z}}+1)\tilde z,\abs{\tilde{z}}^2-\tilde{w}\cdot\bar{\tilde{z}}\right)}{\tan(\dist(z,w))(\tilde{w}\cdot\bar{\tilde{z}}+1)\sqrt{\abs{\tilde{z}}^2+1}}
  \]
  so \eqref{eq:approximate_distance} can be used to show
  \begin{equation}
  \label{eq:approximate_normal}
    \gamma_0'(z,w)=\left(\frac{\tilde w-\tilde z}{\abs{\tilde{z}-\tilde{w}}},0\right)+O(\max\set{\abs{\tilde{z}},\abs{\tilde{w}}}).
  \end{equation}

  In these coordinates, let $v_p=[0:\cdots:0:-i:1]$.  Suppose that for every $j\in\mathbb{N}$ there exists $z_j\in B(p,1/j)$ and $q_j\in\Pi(z_j)$ such that $\left<\gamma'_0(z_j,v_p),\gamma'_0(z_j,q_j)\right>_\omega\geq-A_p$.  We may assume that $[z_j]$ is represented by an element satisfying $z_j^{n+1}=1$, so that $z_j\rightarrow p$.  Observe that $\gamma'_0(z_j,v_p)\rightarrow\gamma'_0(p,v_p)=(0,\ldots,0,-i,0)$.
  Let $u_0$ be a limit point of $\set{\gamma'_0(z_j,q_j)}$, and restrict to a subsequence so that this is the unique limit.  Taking limits, we have $\re(-i\overline{u^n_0})\geq -A_p$, or $\im u^n_0\leq A_p$.  Note that $\abs{u_0}=1$, so $\abs{u_0'}^2+\abs{\re u_0^n}^2=1-\abs{\im u_0^n}^2$.  Furthermore, we must have $\im z_j^n<\im q_j^n$.  Otherwise the interval $(q_j',\re q_j^n+is)$ for $\im q_j^n<s<2\im z_j^n-\im q_j^n$ would lie inside $B(z_j,\delta(z_j))\subset\Omega$ but above the graph of $\varphi$, contradicting \eqref{eq:graph_definition}.  Hence, we must have $\im u_0^n\geq 0$.

  For $s>0$, in our local coordinate patch we define $\tilde{q}_{j,s}'=\tilde{q}_j'-su_0'$, $\re\tilde{q}_{j,s}^n=\re\tilde{q}_j^n-s\re u_0^n$, and $\im\tilde{q}_{j,s}^n=\varphi(\tilde{q}_{j,s}',\re\tilde{q}_{j,s}^n)$.  By assumption,
  \[
    \abs{\im(\tilde{q}^n_{j,s}-\tilde{q}^n_j)}\leq M\sqrt{\abs{\tilde{q}'_{j,s}-\tilde{q}'_j}^2+\abs{\re(\tilde{q}^n_{j,s}-\tilde{q}^n_j)}^2}=Ms\sqrt{1-\abs{\im u_0^n}^2}.
  \]
  This tells us that for every $j$ we can find a limit point $L_j$ of $\set{s^{-1}(\im(\tilde{q}^n_j-\tilde{q}^n_{j,s}))}$ as $s\rightarrow 0^+$ such that $\abs{L_j}\leq M\sqrt{1-\abs{\im u_0^n}^2}$.  Using \eqref{eq:approximate_normal}, there exists a limit point $u_j$ of $\set{\gamma'_0(q_j,q_{j,s})}$ as $s\rightarrow 0^+$ taking the form
  \[
    u_j=-(u_0',\re u_0^n+iL_j,0)\left(1-\abs{\im u_0^n}^2+\abs{L_j}^2\right)^{-1/2}+O(\abs{\tilde{q}_j}).
  \]
  Let $L$ be a limit point of $\set{L_j}$, and further restrict to a subsequence so that $L_j\rightarrow L$.  Since $\tilde{q}_j\rightarrow 0$, we must have
  \begin{equation}
  \label{eq:u_j_limit}
    u_j\rightarrow-(u_0',\re u_0^n+iL,0)\left(1-\abs{\im u_0^n}^2+\abs{L}^2\right)^{-1/2}
  \end{equation}

  Since $B(z_j,\delta(z_j))\subset\Omega$ and $q_{j,s}\in\partial\Omega$, we must have $\dist(z_j,q_{j,s})\geq\dist(z_j,q_j)$.  Using \eqref{eq:point_distance_derivative}, we have
  \[
    0\geq\frac{\dist(q_j,z_j)-\dist(q_{j,s},z_j)}{\dist(q_j,q_{j,s})}\geq\left<\gamma_0'(q_j,q_{j,s}),\gamma_0'(q_j,z_j)\right>_\omega-O(\dist(q_j,q_{j,s})).
  \]
  Considering limit points as $s\rightarrow 0^+$, we see that $0\geq\left<u_j,\gamma_0'(q_j,z_j)\right>_\omega$.  Using \eqref{eq:approximate_normal}, $\gamma_0'(q_j,z_j)\rightarrow -u_0$, so combining this with \eqref{eq:u_j_limit} we see that
  \[
    0\geq\left(1-\abs{\im u_0^n}^2+L\im u_0^n\right)\left(1-\abs{\im u_0^n}^2+\abs{L}^2\right)^{-1/2}.
  \]
  However, since $\abs{L}\leq M\sqrt{1-\abs{\im u_0^n}^2}$,
  \[
    1-\abs{\im u_0^n}^2+L\im u_0^n\geq 1-\abs{\im u_0^n}^2-M\sqrt{1-\abs{\im u_0^n}^2}\abs{\im u_0^n}
  \]
  We know $0\leq\im u_0^n\leq A_p$, so
  \[
    1-\abs{\im u_0^n}^2-M\sqrt{1-\abs{\im u_0^n}^2}\abs{\im u_0^n}\geq \sqrt{1-\abs{\im u_0^n}^2}\left(\sqrt{1-A_p^2}-MA_p\right)
  \]
  Combining inequalities, we find that $0\geq\sqrt{1-A_p^2}-MA_p$, which is equivalent to $A_p\geq \frac{1}{\sqrt{1+M^2}}$, a contradiction.

  We now know that for every $p\in\partial\Omega$, there exists a point $v_p$ and a radius $R_p$ so that for every $z\in B(p,R_p)$ and $q\in\Pi(z)$, $\left<\gamma'_0(z,v_p),\gamma'_0(z,q)\right>_\omega<-A_p$.  We also observe that $v_p$ has been chosen so that $\dist(p,v_p)=\frac{\pi}{4}$.  We may assume that $R_p<\frac{\pi}{4}$.  Choose a finite collection $\set{p_j}$ so that $B(p_j,R_{p_j}/2)$ covers $\partial\Omega$, and let $\set{\chi_j}$ be a smooth partition of unity subordinate to $\set{B(p_j,R_{p_j}/2)}$.

  Note that $z\in\supp\chi_j$ implies that
  \[
    \frac{\pi}{8}<\frac{\pi}{4}-R_{p_j}/2\leq\dist(z,v_{p_j})\leq R_{p_j}/2+\frac{\pi}{4}<\frac{3\pi}{8},
  \]
  so this is uniformly bounded away from $0$ and $\frac{\pi}{2}$.  Fix $d>0$ so that $\chi_j(z)>0$ for at least one $j$ whenever $\delta(z)\leq d$.  For all $z\in\mathbb{CP}^n$ satisfying $\delta(z)\leq d$, we define
  \[
    u(z)=\sum_{\set{j:z\in\supp\chi_j}}\chi_j(z)\gamma'_0(z,v_{p_j})\text{ and }v(z)=\frac{z}{\abs{z}}+\frac{u(z)}{\abs{u(z)}}.
  \]
  We can immediately compute $v(z)\cdot\bar{z}=\abs{z}$ and $\abs{v(z)}^2=2$, so $\dist(z,v(z))=\frac{\pi}{4}$.  Furthermore, since $\dist(z,v_{p_j})$ is uniformly bounded away from $0$ and $\frac{\pi}{2}$ when $z\in\supp\chi_j$, we can check that $v(z)$ is smooth in $z$ using \eqref{eq:geodesic_vector_derivative_q} and \eqref{eq:q_switch_p}.  With this in mind, we may choose $A_0>0$ so that $0<A_0<\abs{u(z)}^{-1}\inf_j A_{p_j}$.

  Now, fix $\eps>0$ satisfying $0<\eps<\abs{u(z)}^{-1}\inf_j A_{p_j}-A_0$.  Since \eqref{eq:lipschitz_geodesic_gradient} implies that $\abs{D_p\gamma_0'(p,v(p))}_\omega\leq 2$, $\gamma_0'(z,q)$ is uniformly Lipschitz in $q$ with a constant independent of $z$ for $z$ sufficiently close to $p$ and $q$ sufficiently close to $v(p)$.  Since $v(z)$ is also Lipschitz there exists $0<R_0<\min\set{d,\inf_j R_{p_j}/2}$ such that for every $p\in\partial\Omega$ and $z\in B(p,R_0)$, $\abs{\gamma'_0(z,v(p))-\gamma'_0(z,v(z))}<\eps$.  Then for $p\in\partial\Omega$, $z\in B(p,R_0)$, and $q\in\Pi(z)$, we have
  \[
    \left<\gamma'_0(z,v(p)),\gamma'_0(z,q)\right>_\omega<\left<\gamma'_0(z,v(z)),\gamma'_0(z,q)\right>_\omega+\eps.
  \]
  Now $\gamma_0'(z,v(z))=\frac{u(z)}{\abs{u(z)}}$, so
  \[
    \left<\gamma'_0(z,v(p)),\gamma'_0(z,q)\right>_\omega
    \leq-\abs{u(z)}^{-1}\sum_{\set{j:z\in\supp\chi_j}}\chi_j(z)A_{p_j}+\eps< -A_0,
  \]
  and we have proven \eqref{eq:vector_inner_product}.

  Since $v$ is actually the restriction of a smooth map to $\partial\Omega$, we can easily see that $D_{p,u}^{\partial\Omega}v(p)=D_{p,u}v(p)$, so the tangential derivative exists and is continuous.

\end{proof}

A critical consequence of \eqref{eq:vector_inner_product} follows when we combine it with \eqref{eq:directional_derivative}.  If $z\in B(p,R_0)\cap\overline\Omega$ and $0\leq t<R_0-\dist(z,p)$, we observe that $\delta(\gamma_t(z,v(p)))$ is an absolutely continuous function of $t$ and compute
\begin{multline}
\label{eq:delta_geodesic_estimate_inward}
  \delta(\gamma_t(z,v(p)))-\delta(z)\\
  =-\int_0^t \sup_{q\in\Pi_\Omega(\gamma_s(p,v(p)))}\left<\gamma'_0(\gamma_s(z,v(p)),v(p)),\gamma'_0(\gamma_s(p,v(p)),q)\right>_\omega ds>tA_0.
\end{multline}
If instead we have $z\in B(p,R_0)\cap\overline\Omega$ and $0\leq t<R_0-\dist(z,p)$ with $\gamma_{-t}(z,v(p))\in\overline\Omega$, we also have
\begin{multline}
\label{eq:delta_geodesic_estimate_outward}
  \delta(z)-\delta(\gamma_{-t}(z,v(p)))\\
  =\int_{\dist(z,v(p))}^t \sup_{q\in\Pi_\Omega(\gamma_{-s}(p,v(p)))}\left<-\gamma'_0(\gamma_{-s}(z,v(p)),v(p)),\gamma'_0(\gamma_{-s}(p,v(p)),q)\right>_\omega ds\\
  >tA_0.
\end{multline}

As a partial converse to Lemma \ref{lem:transverse_vector_field}, we observe that the map $v$ is sufficient to locally parameterize the boundary in terms of a Lipschitz function:
\begin{lem}
\label{lem:projection_to_boundary}
  Let $\Omega\subset\mathbb{CP}^n$ be a Lipschitz domain, and let $v$, $R_0$, and $A_0$ be given by Lemma \ref{lem:transverse_vector_field}.  Then for every $p\in\partial\Omega$ there exists a map $\pi_p:B\left(p,\frac{R_0}{1+A_0^{-1}}\right)\cap\Omega\rightarrow\mathbb{R}$ such that $\delta(z)\leq \pi_p(z)< A_0^{-1}\delta(z)$ and $\gamma_{-\pi_p(z)}(z,v(p))\in\partial\Omega$.  Furthermore, for any $0<A<A_0$, there exists $0<R_A<\frac{R_0}{1+A_0^{-1}}$ such that on $B(p,R_A)$, $\pi_p$ is a Lipschitz map with Lipschitz constant $A^{-1}$.
\end{lem}

\begin{rem}
  If we choose a real hypersurface such that $\gamma_t(z,v(p))$ is transverse to this hypersurface near $p$, then $\pi_p(z)$ allows us to locally express $\partial\Omega$ as a Lipschitz graph over this hypersurface.
\end{rem}

\begin{rem}
  With some additional work, we can show that $\pi_p(z)+\dist(z,v(p))$ is Lipschitz in $z$ with a constant on the order of $\sqrt{A^{-2}-1}$.  On $C^1$ domains we can take the constant $A$ arbitrarily close to $1$ by choosing $R_0$ sufficiently small, so this would allow us to take $\tau$ arbitrarily close to $1$ in Lemma \ref{lem:boundary_sequence} below.  Since we choose to focus on cases where $A$ is very close to $0$, we omit the additional computations necessary for this refinement.
\end{rem}

\begin{proof}
  Let $z\in B\left(p,\frac{R_0}{1+A_0^{-1}}\right)\cap\Omega$.  This implies $\delta(z)<\frac{R_0}{1+A_0^{-1}}$ as well, so $\dist(z,p)+A_0^{-1}\delta(z)<R_0$.  If $t$ satisfies $R_0-\dist(z,p)>t\geq A_0^{-1}\delta(z)$ and $\gamma_{-t}(z,v(p))\in\overline\Omega$, we would have a contradiction with \eqref{eq:delta_geodesic_estimate_outward}, so $\gamma_{-t}(z,v(p))\notin\overline\Omega$ whenever $t\geq A_0^{-1}\delta(z)$.  Hence, there must exist $\pi_p(z)$ satisfying $\delta(z)\leq \pi_p(z)<A_0^{-1}\delta(z)$ such that $\gamma_{-\pi_p(z)}(z,v(p))\in\partial\Omega$.

  By \eqref{eq:lipschitz_geodesic_p}, the derivative of $\gamma_{-t}(z,v(p))$ with respect to $z$ is bounded by $1$ when $z=p$ and $t=0$, so we can choose $0<R_A\leq\frac{R_0}{1+A_0^{-1}}$ so that $\gamma_{-t}(z,v(p))$ has a Lipschitz constant of $A_0 A^{-1}$ in $z$ for $z\in B(p,R_A)$ and $\delta(z)\leq t\leq A_0^{-1}\delta(z)$.  For $z,w\in B(p,R_A)\cap\Omega$, we may assume that $\pi_p(w)>\pi_p(z)$.  Let $\hat{w}=\gamma_{-\pi_p(z)}(w,v(p))$ and $\hat{z}=\gamma_{-\pi_p(z)}(z,v(p))$.  Note that $\hat{z}\in\partial\Omega$, and $\dist(\hat{z},\hat{w})\leq A_0 A^{-1}\dist(z,w)$.  Hence $\delta(\hat{w})\leq A_0 A^{-1}\dist(z,w)$.  However, \eqref{eq:delta_geodesic_estimate_outward} implies $\delta(\hat{w})>(\pi_p(w)-\pi_p(z))A_0$, so we have $\abs{\pi_p(w)-\pi_p(z)}\leq  A^{-1}\dist(z,w)$.
\end{proof}

When we work locally, the geometric quantities of greatest interest will be the distance to the boundary and the distance to a geodesic transverse to the boundary, so it will make sense to define our local neighborhoods in terms of these quantities.  However, we will still need to know that these new neighborhoods are uniformly comparable to the usual geodesic balls, as shown in the following lemma.

\begin{lem}
\label{lem:neighborhood_base}
  Let $\Omega\subset\mathbb{CP}^n$ be a Lipschitz domain and let $v$, $R_0$, and $A_0$ be given by Lemma \ref{lem:transverse_vector_field}.  For every $p\in\partial\Omega$, let $\Gamma(p,S,T)$ denote the connected component of
  \[
    \set{z\in\overline\Omega:\dist(z,\gamma(p,v(p)))\leq S\text{ and }\delta(z)\leq T}
  \]
  containing $p$.  Given $0<A<A_0$ there exists $0<R_A<R_0$ such that $\gamma(p,S,T)\subset B(p,R)$ whenever $S,T>0$ and $(1+A^{-1})S+A^{-1}T\leq R<R_A$.
\end{lem}

\begin{proof}
  Choose $A<\tilde{A}<A_0$ and let $\pi_p$ and $R_{\tilde{A}}$ be given by Lemma \ref{lem:projection_to_boundary} so that $\pi_p$ has Lipschitz constant ${\tilde{A}}^{-1}$ on $B(p,R_{\tilde{A}})$.  By \eqref{eq:lipschitz_geodesic_p}, the derivative of $\gamma_t(z,v(p))$ with respect to $z$ is bounded by $1$ whenever $z=p$ and $t=0$, so we may also choose $0<R_A<R_{\tilde{A}}$ so that $\gamma_t(z,v(p))$ has Lipschitz constant $\tilde{A}A^{-1}$ with respect to $z$ whenever $z\in B(p,R_A)$ and $\abs{t}\leq R_A$.  For $0<R<R_A$, choose $S<\frac{AR}{1+A}$ and $T\leq AR-(A+1)S$.  Observe that the inequality characterizing $S$ is equivalent to $A^{-1}S<R-S$ and the inequality characterizing $T$ is equivalent to $A^{-1}(S+T)\leq R-S$.  Let $z\in\Gamma(p,S,T)\cap B(p,R_A)$, and let $t_0$ denote the value of $t$ minimizing the distance from $z$ to $\gamma_t(p,v(p))$.  Note that $\abs{t_0}<R_A$.

  If $t_0<0$, then since $z\in B(p,R_A)$ and $\gamma_{t_0}(v(p),p)\in B(p,R_A)$, we use the Lipschitz property of $\gamma_{-t_0}(\cdot,v(p))$ to obtain
  \[
    \dist(\gamma_{-t_0}(z,v(p)),p)=\dist(\gamma_{-t_0}(z,v(p)),\gamma_{-t_0}(\gamma_{t_0}(p,v(p)),v(p)))\leq \tilde{A}A^{-1}S.
  \]
  Since $\pi_p(p)=0$, we can couple this with the Lipschitz property of $\pi_p$ to obtain $\pi_p(\gamma_{-t_0}(z,v(p)))\leq A^{-1}S$.  We know $\pi_p(\gamma_{-t_0}(z,v(p)))=\pi_p(z)-t_0$, so $\pi_p(z)-t_0\leq A^{-1}S$.  Since $\pi_p(z)\geq 0$, we have $-t_0\leq A^{-1}S<R-S$.

  If $t_0\geq 0$, we note that $\pi_p(\gamma_{t_0}(p,v(p)))=t_0$, so we again use the Lipschitz property of $\pi_p$ to estimate $\abs{\pi_p(z)-t_0}\leq \tilde{A}^{-1}S$.  From Lemma \ref{lem:projection_to_boundary} we know $\pi_p(z)<A_0^{-1}\delta(z)\leq A_0^{-1}T$, so $t_0\leq \tilde{A}^{-1}S+A_0^{-1}T<A^{-1}(S+T)\leq R-S$.

  Combining both cases, we have $\abs{t_0}<R-S$.  Hence, $\dist(z,p)\leq S+\abs{t_0}<R$.  Since $\Gamma(p,S,T)\cap B(p,R_A)\subset B(p,R)$ and $\Gamma(p,S,T)$ was defined to be connected, we conclude that $\Gamma(p,S,T)\subset B(p,R)$.

\end{proof}

We will also need to show that the map $\gamma_t(p,v(p)):(0,\eps)\times\partial\Omega\rightarrow\Omega$ is invertible near the boundary for $\eps>0$ sufficiently small.  On $C^1$ domains this would follow from the implicit function theorem, but this does not apply on Lipschitz domains, so we will have to find a suitable substitute by careful analysis of the properties of $v$.

\begin{lem}
\label{lem:surjection}
  Let $\Omega\subset\mathbb{CP}^n$ be a Lipschitz domain and let $v$ be given by Lemma \ref{lem:transverse_vector_field}.  Then there exists $T_1>0$ such that for every $z\in\overline\Omega$ satisfying $\delta(z)\leq T_1$ there exists a unique $p\in\partial\Omega$ and $0\leq t<A_0^{-1}\delta(z)$ such that $[\gamma_t(p,v(p))]=[z]$.
\end{lem}

\begin{proof}
  We begin with the proof of uniqueness.  Suppose that for every $j\in\mathbb{N}$ there exists $z_j\in\Omega$ such that $\delta(z_j)\leq\frac{1}{j}$ and there exist $0<t_j\leq A_0^{-1}\delta(z_j)$ and $p_j,q_j\in\partial\Omega$ such that $[p_j]\neq [q_j]$ and $[z_j]=[\gamma_{t_j}(p_j,v(p_j))]=[\gamma_{t_j}(q_j,v(q_j))]$.  Set $x_j+iy_j=1-\gamma_0'(z_j,v(p_j))\cdot\overline{\gamma_0'(z_j,v(q_j))}$ for $x_j,y_j\in\mathbb{R}$, and note that $0\leq x_j\leq 2$ and $\abs{y_j}\leq\sqrt{2x_j-x_j^2}$.  Observe that
  \begin{multline*}
    \gamma_{-t_j}(z_j,v(p_j))\cdot\overline{\gamma_{-t_j}(z_j,v(q_j))}=\cos^2 t_j+\sin^2 t_j\gamma_0'(z_j,v(p_j))\cdot\overline{\gamma_0'(z_j,v(q_j))}\\
    =1-\sin^2 t_j(x_j+iy_j),
  \end{multline*}
  so
  \begin{multline*}
    \sin^2(\dist(p_j,q_j))=\sin^2(\dist(\gamma_{-t_j}(z_j,v(p_j)),\gamma_{-t_j}(z_j,v(q_j))))\\
    =1-\abs{1-\sin^2 t_j(x_j+iy_j)}^2
    =2x_j\sin^2 t_j-(x_j^2+y_j^2)\sin^4 t_j.
  \end{multline*}
  Since $\cos^2(\pi/4-t_j)=\frac{1}{2}+\sin t_j\cos t_j$ and $\sin^2(\pi/4-t_j)=\frac{1}{2}-\sin t_j\cos t_j$,
  \begin{multline*}
    \gamma_{\pi/4-t_j}(z_j,v(p_j))\cdot\overline{\gamma_{\pi/4-t_j}(z_j,v(q_j))}\\
    =\left(\frac{1}{2}+\sin t_j\cos t_j\right)+\left(\frac{1}{2}-\sin t_j\cos t_j\right)(1-x_j-iy_j)\\
    =1-\frac{x_j+iy_j}{2}+\sin t_j\cos t_j(x_j+iy_j),
  \end{multline*}
  so
  \begin{multline*}
    \sin^2(\dist(v(p_j),v(q_j)))=\sin^2(\dist(\gamma_{\pi/4-t_j}(z_j,v(p_j)),\gamma_{\pi/4-t_j}(z_j,v(q_j))))\\
    =1-\abs{1-\frac{x_j+iy_j}{2}+\sin t_j\cos t_j(x_j+iy_j)}^2\\
    =x_j-\frac{1}{4}(x_j^2+y_j^2)
    -\sin t_j\cos t_j\left(2x_j-x_j^2-y_j^2\right)
    -\sin^2 t_j\cos^2 t_j(x_j^2+y_j^2).
  \end{multline*}
  Since $t_j\rightarrow 0$, $\dist(p_j,q_j)\rightarrow 0$ as well.  Since $v$ is Lipschitz, we may assume that there exists $M>0$ such that $\sin^2(\dist(v(p_j),v(q_j)))\leq M\sin^2(\dist(p_j,q_j))$ for sufficiently large $j$.  Hence,
  \begin{multline*}
    \frac{1}{4}(x_j^2+y_j^2)-x_j+\sin t_j\cos t_j\left(2x_j-x_j^2-y_j^2\right)+2Mx_j\sin^2 t_j\\
    +\sin^2 t_j\cos^2 t_j(x_j^2+y_j^2)-M(x_j^2+y_j^2)\sin^4 t_j\geq 0.
  \end{multline*}
  The coefficient of $y_j^2$ will be positive for sufficiently large $j$, so we can substitute $\abs{y_j}\leq\sqrt{2x_j-x_j^2}$ and obtain
  \[
    -\frac{1}{2}x_j+2Mx_j\sin^2 t_j+2\sin^2 t_j\cos^2 t_jx_j-2Mx_j\sin^4 t_j\geq 0.
  \]
  If $x_j>0$, we can divide by $x_j$ and obtain a contradiction for sufficiently large $j$.  Hence $x_j=y_j=0$ for all sufficiently large $j$.  However, this implies $\dist(p_j,q_j)=0$, another contradiction.  We conclude that there exists $T_0>0$ such that if $z\in\overline\Omega$ satisfies $\delta(z)\leq T_0$ and there exist $p\in\Omega$ and $0\leq t<A_0^{-1}\delta(z)$ such that $[\gamma_t(p,v(p))]=[z]$, then no other such $t$ and $p$ exist.

  Turning to our existence proof, let $A_0$ be given by Lemma \ref{lem:transverse_vector_field} and for $0<A<A_0$ let $R_A$ be given by Lemma \ref{lem:projection_to_boundary}.  Choose any $T_1<\min\set{T_0/2,A_0 R_A/2}$.

  Suppose that there exists $z\in\Omega$ such that $[z]$ is not in the range of $[\gamma_t(p,v(p))]$ and $\delta(z)\leq T_1$.  Let $w$ be the closest point to $z$ in $\overline\Omega$ that does lie in the range of $[\gamma_t(p,v(p))]$ for $p\in\partial\Omega$ and $0\leq t\leq\frac{\pi}{2}$.  Since $[\gamma_0(p,v(p))]=[p]$, every point in the boundary lies in the range of $\gamma_t(p,v(p))$, so $\dist(w,z)\leq\delta(z)$.  Therefore, $\delta(w)\leq 2\delta(z)$.  Fix $q\in\partial\Omega$ and $0\leq s\leq\frac{\pi}{2}$ so that $[w]=[\gamma_s(q,v(q))]$.  We may choose a representative element for $[q]$ to satisfy $q=\gamma_{-s}(w,v(q))$.  From \eqref{eq:delta_geodesic_estimate_outward}, $s<A_0^{-1}\delta(w)$, so $s<2A_0^{-1}\delta(z)\leq 2A_0^{-1}T_1$.

  For $j\in\mathbb{N}$, let $d_j=\dist(z,w)/j$ and set $z_j=\gamma_{d_j}(w,z)$, so that $[z_j]\rightarrow [w]$.  Observe that
  \[
    \delta(z_j)\leq\delta(z)+(j-1)d_j\leq (2j-1)\delta(z)/j
  \]
  and
  \[
    \dist(z_j,q)\leq s+d_j<(2A_0^{-1}+1/j)\delta(z).
  \]
  For $j$ sufficiently large, we have $\dist(z_j,q)<2A_0^{-1}T_1<R_A$.  Hence, Lemma \ref{lem:projection_to_boundary} gives us $\pi_q(z_j)$ with Lipschitz constant $A^{-1}$ satisfying $\delta(z_j)\leq \pi_q(z_j)<A_0^{-1}\delta(z_j)$ and $\gamma_{-\pi_q(z_j)}(z_j,v(q))\in\partial\Omega$.  Set $s_j=\pi_q(z_j)$ and $q_j=\gamma_{-s_j}(z_j,v(q))$.  By Lemma \ref{lem:projection_to_boundary}, $\abs{s_j-s}\leq A^{-1}d_j$.  Using \eqref{eq:lipschitz_geodesic_p}, $\gamma_{-s}(\cdot,v(q))$ is Lipschitz, so $\dist(q_j,q)\leq O(d_j)$.  Furthermore, $[q_j]\neq[q]$, since otherwise $[z_j]=[\gamma_{s_j}(q,v(q))]$, contradicting the assumption that $w$ is the closest point to $z$ with this property.

  Set $w_j=\gamma_s(q_j,v(q_j))$.  We know that $v$ is Lipschitz, so since Corollary \ref{cor:lipschitz_geodesic} implies that both $\gamma_s(\cdot,v(q_j))$ and $\gamma_s(q_j,\cdot)$ are Lipschitz, we have $\dist(w_j,w)\leq O(d_j)$.  Since $[q_j]\neq [q]$ and $\delta(w)\leq T_0$, we may use our uniqueness result to show $[w_j]\neq [w]$ for sufficiently large $j$.  Furthermore, since $[z_j]=[\gamma_{s_j}(q_j,v(q))]$, we must also have $\dist(z_j,w_j)<O(d_j)$.  By assumption $\dist(z,w_j)\geq\dist(z,w)$.  Using \eqref{eq:point_distance_derivative}, we have
  \[
    0\leq\frac{\dist(w_j,z)-\dist(w,z)}{\dist(w,w_j)}\leq-\left<\gamma_0'(w,w_j),\gamma_0'(w,z)\right>_\omega+O(d_j).
  \]
  Hence, $\left<\gamma_0'(w,w_j),\gamma_0'(w,z)\right>_\omega\leq O(d_j)$.  However, we know $\left<\gamma_0'(w,z_j),\gamma_0'(w,z)\right>_\omega=1$, so since $\dist(z_j,w_j)<O(d_j)$, we are left with $1\leq O(d_j)$, which is a contradiction for sufficiently large $j$.

\end{proof}

Our last result in this section shows that we can always find tangent vectors in a given direction that are uniformly transverse to $\gamma_t(p,v(p))$.

\begin{lem}
\label{lem:boundary_sequence}
  Let $\Omega\subset\mathbb{CP}^n$ be a Lipschitz pseudoconvex domain.  Let $v$ and $A_0$ be as in Lemma \ref{lem:transverse_vector_field}.  Then for every $p,q\in\partial\Omega$ satisfying $0<\dist(p,q)<\frac{\pi}{2}$, we have $u\in\Tan(\partial\Omega,p)$ and real constants $\nu$ and $\tau$ such that $A_0\abs{\nu}\leq\tau$ and
  \[
  \label{eq:boundary_sequence_estimate}
    u=\nu\gamma_0'(p,v(p))+\tau\gamma_0'(p,q).
  \]
  In particular, $\tau\geq\frac{A_0}{1+A_0}$.
\end{lem}

\begin{proof}
  Using \eqref{eq:delta_geodesic_estimate_inward}, we have $\delta(\gamma_t(p,v(p)))>tA_0$ for $0\leq t<R_0$, so
  \[
    \delta(\gamma_s(\gamma_t(p,v(p)),q))>tA_0-s
  \]
  for $0\leq s\leq tA_0$.  For $0<A<A_0$, let $\pi_p$ and $R_{A}$ be given by Lemma \ref{lem:projection_to_boundary}.  For $j\in\mathbb{N}$, let $t_j=\frac{R_{A}}{j(1+A_0)}$ and $s_j=A_0 t_j$, then set $z_j=\gamma_{s_j}(\gamma_{t_j}(p,v(p)),q)$.  We have $z_j\in\overline\Omega$ and $\dist(z_j,p)\leq s_j+t_j\leq R_{A}/j$, so we can define $p_j=\gamma_{-\pi_p(z_j)}(z_j,v(p))\in\partial\Omega$.  Since $A^{-1}$ is the Lipschitz constant for $\pi_p$, we must have $\abs{\pi_p(z_j)-t_j}\leq A^{-1} s_j$.

  Corollary \ref{cor:lipschitz_geodesic} implies $\gamma_s(\cdot,q)$ and $\gamma_t(\cdot,v(p))$ are both Lipschitz functions, provided that the distances to $q$ and $v(p)$ are uniformly bounded away from $0$ and $\frac{\pi}{2}$.  With this in mind, we compute
  \[
    \gamma_{t_j}(p,v(p))=\frac{p}{\abs{p}}+t_j\gamma_0'(p,v(p))+O(j^{-2}),
  \]
  so
  \[
    z_j=\frac{p}{\abs{p}}+t_j\gamma_0'(p,v(p))+s_j\gamma_0'(p,q)+O(j^{-2}).
  \]
  Since $\abs{\pi_p(z_j)}\leq O(j^{-1})$ and $\abs{\gamma_0'(p,v(p))-\gamma_0'(z_j,v(p))}\leq O(j^{-1})$, we have
  \begin{multline*}
    p_j=z_j-\pi_p(z_j)\gamma_0'(p,v(p))+O(j^{-2})\\
    =\frac{p}{\abs{p}}+(t_j-\pi_p(z_j))\gamma_0'(p,v(p))+s_j\gamma_0'(p,q)+O(j^{-2}),
  \end{multline*}
  so $\dist(p,p_j)\leq O(j^{-1})$ and
  \[
    \gamma_0'(p,p_j)=\frac{(t_j-\pi_p(z_j))\gamma_0'(p,v(p))+s_j\gamma_0'(p,q)}{\dist(p,p_j)}+O(j^{-1}).
  \]
  Since this vector has length one, $\set{\frac{t_j-\pi_p(z_j)}{\dist(p,p_j)}}$ and $\set{\frac{s_j}{\dist(p,p_j)}}$ must both be bounded.  Hence, we can restrict to a subsequence on which $\frac{t_j-\pi_p(z_j)}{\dist(p,p_j)}\rightarrow\nu$ and $\frac{s_j}{\dist(p,p_j)}\rightarrow\tau$ for some constants $\nu$ and $\tau$ satisfying $\abs{\nu}\leq A^{-1}\tau$.  Repeating the argument for a sequence $A_k$ increasing towards $A_0$ and using a diagonalization argument to extract a new subsequence, we conclude that we may assume $\abs{\nu}\leq A_0^{-1}\tau$.

  Furthermore, since $\gamma_0'(p,p_j)$ has unit length, we have
  \[
    \nu^2+2\nu\tau\left<\gamma_0'(p,v(p)),\gamma_0'(p,q)\right>_\omega+\tau^2=1.
  \]
  Solving for $\nu$, we find
  \[
    \nu=-\tau\left<\gamma_0'(p,v(p)),\gamma_0'(p,q)\right>_\omega\pm\sqrt{1-\tau^2(1-\left<\gamma_0'(p,v(p)),\gamma_0'(p,q)\right>_\omega^2)}.
  \]
  Since $\sqrt{1-\tau^2(1-\left<\gamma_0'(p,v(p)),\gamma_0'(p,q)\right>_\omega^2)}\geq\tau\abs{\left<\gamma_0'(p,v(p)),\gamma_0'(p,q)\right>_\omega}$, we have
  \[
    \abs{\nu}\geq\sqrt{1-\tau^2(1-\left<\gamma_0'(p,v(p)),\gamma_0'(p,q)\right>_\omega^2)}-\tau\abs{\left<\gamma_0'(p,v(p)),\gamma_0'(p,q)\right>_\omega}.
  \]
  Since the lower bound is decreasing with respect to $\abs{\left<\gamma_0'(p,v(p)),\gamma_0'(p,q)\right>_\omega}$, we have $\abs{\nu}\geq 1-\tau$.  Combined with $\abs{\nu}\leq A_0^{-1}\tau$, we have $\tau\geq\frac{A_0}{1+A_0}$.  Let $u=\lim\gamma_0'(p,p_j)$.

\end{proof}

\section{Proof of Main Theorem}
\label{sec:proof_main_theorem}

Now, we combine the results of the previous sections in the form that will be the most helpful in the proof of our main theorem:

\begin{lem}
\label{lem:tools}
  Let $\Omega\subset\mathbb{CP}^n$ be a Lipschitz pseudoconvex domain.  There exist constants $A,B>0$ and $S,T>0$ such that for every $p\in\partial\Omega$ there exists a closed connected set $\Gamma_p$ containing $p$, a real-valued function $\mu_p(z)$ on $\Gamma_p$, and a holomorphic isometry $\phi_t^p(z)$ on $\mathbb{CP}^n$ with the following properties:
  \begin{enumerate}
    \item \label{item:t_kahler_generator} $\omega=i\ddbar\mu_p$ on $\Gamma_p$.

    \item \label{item:t_continuity} $\mu_p(z)$ depends continuously on $p$.

    \item \label{item:t_gamma_characterization} $0<\delta(z)<T$ and $0\leq\mu_p(z)\leq\log\sec S$ in the interior of $\Gamma_p$, and either $\delta(z)=0$, $\delta(z)=T$, or $\mu_p(z)=\log\sec S$ on the boundary of $\Gamma_p$.

    \item \label{item:t_surjection} For $z\in\Omega$ satisfying $\delta(z)\leq T$, there exists $p\in\partial\Omega$ satisfying $z\in\Gamma_p$.
    
    \item \label{item:t_phi_domain} For $z\in\Gamma_p$ and $0\leq t\leq T-\delta(z)$, $\phi_t^p(z)\in\Gamma_p$.

    \item \label{item:t_delta_derivative_estimate} For all $z\in\Gamma_p$ and $0<t<T-\delta(z)$,
    \begin{equation}
    \label{eq:delta_derivative_estimate}
      A\leq D^+_t\delta(\phi_{t}^p(z))\leq D^-_t\delta(\phi_t^p(z))\leq 1.
    \end{equation}

    \item \label{item:t_delta_estimate} For all $z\in\Gamma_p$ and $0\leq t\leq T-\delta(z)$,
    \begin{equation}
    \label{eq:delta_estimate}
      \delta(z)+At\leq\delta(\phi_t^p(z))\leq\delta(z)+t.
    \end{equation}

    \item \label{item:t_level_curves} $\mu_p(\phi_t^p(z))=\mu_p(z)$ for all $z\in\Gamma_p$ and $0\leq t\leq T-\delta(z)$.

    \item \label{item:t_Lipschitz_in_p} If $p,q\in\partial\Omega$ and $z\in\Gamma_p\cap\Gamma_q$, then
    \begin{equation}
    \label{eq:Lipschitz_in_p}
      \dist(\phi_t^p(z),\phi_t^q(z))\leq Bt\dist(p,q).
    \end{equation}

    \item \label{item:t_boundary_sequence} If $p\in\partial\Omega$ and $z\in\Gamma_p$ satisfy $\mu_p(z)=\log\sec S$, then there exists $u\in\Tan(\partial\Omega,p)$ such that $D_{p,u}\mu_p(z)\leq -\frac{SA^3}{\sqrt{2}(1+A)^3}$.

    \item \label{item:t_Lipschitz_in_z} $\mu_p(z)$ has Lipschitz constant $1$ in $z$.
  \end{enumerate}
\end{lem}

\begin{proof}
  Let $v$, $A_0$, and $R_0$ be given by Lemma \ref{lem:transverse_vector_field}.  Let $\mu_{p,v(p)}$ and $\phi_t^{p,v(p)}$ be given by Lemma \ref{lem:weight_and_map}.

  Let $\mu_p=\mu_{p,v(p)}$.  We obtain \eqref{item:t_kahler_generator} from Lemma \ref{lem:weight_and_map} \eqref{item:w_kahler_generator} and \eqref{item:t_continuity} since $v$ is Lipschitz and Lemma \ref{lem:weight_and_map} \eqref{item:w_first_derivative} implies that $\mu$ is $C^1$ in $p$ and $q$.

  Pick any $0<A<A_0$.  From Lemma \ref{lem:weight_and_map} \eqref{item:m_special_cases}, $(\phi_t^{p,v(p)})'(z)$ can be made arbitrarily close to $\gamma_t'(z,v(p))$ by requiring $z$ to be close to $p$.  Combining this with \eqref{eq:directional_derivative} and \eqref{eq:vector_inner_product}, we can choose $\eps>0$ sufficiently small so that there exists $R_1>0$ with the property that if $\phi_{(1-\eps)t}^{p,v(p)}(z)\in B(p,R_1)$, we have
  \begin{equation}
  \label{eq:preliminary_delta_derivative_estimate}
    A\leq D^+_t\delta\left(\phi_{(1-\eps)t}^{p,v(p)}(z)\right)
    \leq D^-_t\delta\left(\phi_{(1-\eps)t}^{p,v(p)}(z)\right)\leq 1.
  \end{equation}
  Using Lemma \ref{lem:neighborhood_base}, we may choose $0<S_1<\frac{AR_1}{1+A}$ and $0<T_1\leq AR_1-(1+A)S_1$ so that \eqref{eq:preliminary_delta_derivative_estimate} holds on $\Gamma(p,S_1,T_1)$.  Assuming $R_1<\frac{\pi}{4}$, then Lemma \ref{lem:weight_and_map} \eqref{item:w_domain} implies $\Gamma(p,S_1,T_1)\subset\dom\mu_{p,v(p)}$.  Set $\phi_t^p(z)=\phi_{(1-\eps)t}^{p,v(p)}(z)$.  For $0<S\leq S_1$ and $0<T\leq T_1$ to be chosen later, let $\Gamma_p$ be the connected component of
  \[
    \set{z\in\overline\Omega:\mu_p(z)\leq \log\sec S\text{ and }\delta(z)\leq T}
  \]
  containing $p$.  Using Lemma \ref{lem:weight_and_map} \eqref{item:w_lower_bound}, this is a subset of $\Gamma(p,S,T)$, and we have \eqref{item:t_gamma_characterization} by definition.  Lemma \ref{lem:surjection} implies \eqref{item:t_surjection} since there exists $p\in\partial\Omega$ such that $z=\gamma_t(p,v(p))$ and hence $\mu_p(z)=0$.

  Integrating \eqref{eq:preliminary_delta_derivative_estimate} gives us \eqref{eq:delta_estimate}, provided that $z$ and $\phi_t^p(z)$ are both in $\Gamma_p$. However, if $\delta(\phi_t^p(z))=T$, then \eqref{eq:delta_estimate} will imply that $\delta(z)+t\geq T$.  Hence, if $z\in\Gamma_p$ and $0\leq t\leq T-\delta(z)$, we can conclude that $\phi_t^p(z)\in\Gamma_p$ as well, implying \eqref{item:t_phi_domain}.  With this technicality taken care of, we immediately obtain \eqref{eq:delta_derivative_estimate} from \eqref{eq:preliminary_delta_derivative_estimate}, and \eqref{eq:delta_estimate} will follow without qualification.  Furthermore, we can now obtain \eqref{item:t_level_curves} from Lemma \ref{lem:weight_and_map} \eqref{item:m_level_curves}.  Since $v$ is known to be Lipschitz, we obtain \eqref{eq:Lipschitz_in_p} from Lemma \ref{lem:weight_and_map} \eqref{item:m_Lipschitz_in_p}.

  To prove \eqref{item:t_boundary_sequence}, we suppose that for every $j\in\mathbb{N}$ if we set $R_j=\frac{R_1}{j}$, $T_j=\frac{A R_j}{j^2}$, and $S_j=\frac{A R_j-T_j}{1+A}$, then there exists $p_j\in\partial\Omega$ and $z_j\in\Gamma_{p_j}$ such that $\mu_{p_j}(z_j)=\log\sec S_j$ and for every $u\in\Tan(\partial\Omega,p_j)$ we have $D_{p_j,u}\mu_{p_j}(z_j)>-\frac{S_j A^3}{\sqrt{2}(1+A)^3}$.  By Lemma \ref{lem:neighborhood_base}, $\Gamma_{p_j}\subset B(p_j,R_j)$.

  For $j$ sufficiently large, Lemma \ref{lem:surjection} implies that there exists $q_j\in\partial\Omega$ and $0\leq t_j<A_0^{-1}\delta(z_j)$ such that $[z_j]=[\gamma_{t_j}(q_j,v(q_j))]$.  By Lemma \ref{lem:weight_and_map} \eqref{item:w_special_case}, $\mu_{q_j}(q_j)=0$.  We have $\dist(p_j,q_j)\leq R_j+t_j\leq R_j+O(T_j)$.  From Lemma \ref{lem:weight_and_map} \eqref{item:m_level_curves} and Lemma \ref{lem:weight_and_map} \eqref{item:m_special_cases}, $\frac{d}{dt}\mu_{p_j}(\gamma_t(p_j,v(p_j)))=0$, so given any $M>0$ there exists $R_M>0$ so that $\mu_{p_j}(\gamma_t(z,v(z)))$ has Lipschitz constant $M>0$ in $t$ whenever $z\in B(p_j,R_M)$.  Note that Lemma \ref{lem:weight_and_map} \eqref{item:w_invariance} guarantees that $R_M$ can be chosen independently of $j$.  Hence, for $j$ sufficiently large we will have
  \[
    \abs{\mu_{p_j}(z_j)-\mu_{p_j}(q_j)}\leq Mt_j<MA_0^{-1}T_j.
  \]
  This implies that $\mu_{p_j}(q_j)\geq\log\sec S_j-O(T_j)$.  Since $\abs{\log\sec S-\frac{1}{2}S^2}\leq O(S^3)$,
  \begin{equation}
  \label{eq:mu_p_lower_bound}
    \liminf_{j\rightarrow\infty}\frac{\mu_{p_j}(q_j)}{(\dist(p_j,q_j))^2}\geq\liminf_{j\rightarrow\infty}\frac{\log\sec S_j-O(T_j)}{R_j^2+O(T_jR_j)}=\frac{A^2}{2(1+A)^2}
  \end{equation}
  On the other hand, \eqref{eq:w_second_derivative} implies
  \begin{equation}
  \label{eq:mu_p_limit}
    \lim_{j\rightarrow\infty}\abs{\frac{\mu_{p_j}(q_j)}{(\dist(p_j,q_j))^2}-\frac{1}{2}\left(1-\re\left(\left(\gamma_0'(p_j,q_j)\cdot\overline{\gamma_0'(p_j,v(p_j))}\right)^2\right)\right)}=0.
  \end{equation}
  Once again, we use Lemma \ref{lem:weight_and_map} \eqref{item:w_invariance} to guarantee that the convergence is independent of $p_j$.  Since the second term in \eqref{eq:mu_p_limit} is bounded by $1$, we have
  \begin{equation}
  \label{eq:mu_p_bounds}
    \limsup_{j\rightarrow\infty}\frac{S_j^2}{2(\dist(p_j,q_j))^2}\leq\limsup_{j\rightarrow\infty}\frac{\mu_{p_j}(q_j)}{(\dist(p_j,q_j))^2}\leq 1.
  \end{equation}
  We may also use \eqref{eq:mu_p_limit} to show
  \[
    \liminf_{j\rightarrow\infty}\frac{\mu_{p_j}(q_j)}{(\dist(p_j,q_j))^2}\leq\liminf_{j\rightarrow\infty}\frac{1}{2}\left(1-\re\left(\left(\gamma_0'(p_j,q_j)\cdot\overline{\gamma_0'(p_j,v(p_j))}\right)^2\right)\right),
  \]
  so \eqref{eq:mu_p_lower_bound} implies 
  \begin{equation}
  \label{eq:liminf_lower_bound}
    \frac{A^2}{(1+A)^2}\leq\liminf_{j\rightarrow\infty}\left(1-\re\left(\left(\gamma_0'(p_j,q_j)\cdot\overline{\gamma_0'(p_j,v(p_j))}\right)^2\right)\right).
  \end{equation}

  Let $u_j\in\Tan(\partial\Omega,p_j)$ satisfying $u_j=\nu_j\gamma_0'(p_j,v(p_j))+\tau_j\gamma_0'(p_j,q_j)$ be given by Lemma \ref{lem:boundary_sequence}, where $\tau_j\geq\frac{A_0}{1+A_0}$.  Since $v$ is $C^1$, $D_{p_j,u_j}v(p_j)$ is bounded, so we can apply \eqref{eq:w_first_derivative} with $u^{p_j}=u_j$ and $u^{v(p_j)}=D_{p_j,u_j}v(p_j)$.  Since $\nu_j$ is real, terms involving $\nu_j\gamma_0'(p_j,v(p_j))$ will cancel in \eqref{eq:w_first_derivative}, giving us
  \begin{multline*}
    D_{p_j,u_j}\mu_{p_j}(z_j)=-\dist(z_j,p_j)\tau_j\\
    \times\re\left(\gamma_0'(p_j,z_j)\cdot\overline{\gamma_0'(p_j,q_j)}-\gamma_0'(p_j,z_j)\cdot\overline{\gamma_0'(p_j,v(p_j))}\gamma_0'(p_j,q_j)\cdot\overline{\gamma_0'(p_j,v(p_j))}\right)\\
    +O((\dist(z,p_j))^2),
  \end{multline*}
  with Lemma \ref{lem:weight_and_map} \eqref{item:w_invariance} guaranteeing that the constant in the error term is independent of $j$.  Since \eqref{eq:mu_p_bounds} implies $\dist(p_j,q_j)\geq O(R_j)$, \eqref{eq:lipschitz_geodesic_gradient} implies that $\gamma_0'(p_j,z_j)$ has a Lipschitz constant in $z_j$ on the order of $R_j^{-1}$.  Since $\dist(z_j,q_j)\leq O(R_j/j^2)$ we have
  \[
    \lim_{j\rightarrow\infty}\abs{\gamma_0'(p_j,z_j)-\gamma_0'(p_j,q_j)}=0.
  \]
  Hence,
  \[
    \limsup_{j\rightarrow\infty}\frac{D_{p_j,u_j}\mu_{p_j}(z_j)}{\dist(z_j,p_j)}\leq-\frac{A_0}{1+A_0}\liminf_{j\rightarrow\infty}\left(1-\re\left(\gamma_0'(p_j,q_j)\cdot\overline{\gamma_0'(p_j,v(p_j))}\right)^2\right).
  \]
  Combining this with our assumed lower bound for $D_{p_j,u_j}\mu_{p_j}(z_j)$, \eqref{eq:mu_p_bounds}, and \eqref{eq:liminf_lower_bound} gives us
  \begin{multline*}
    -\frac{A^3}{(1+A)^3}\leq\limsup_{j\rightarrow\infty}-\frac{S_j A^3}{\sqrt{2}(1+A)^3\dist(z_j,p_j)}\\
    \leq\limsup_{j\rightarrow\infty}\frac{D_{p_j,u_j}\mu_{p_j}(z_j)}{\dist(z_j,p_j)}\leq-\frac{A_0 A^2}{(1+A_0)(1+A)^2},
  \end{multline*}
  a contradiction.  Therefore, \eqref{item:t_boundary_sequence} follows.

  Regarding Lipschitz constants in $z$, observe that Lemma \ref{lem:weight_and_map} \eqref{item:w_lower_bound} and Lemma \ref{lem:weight_and_map} \eqref{item:w_special_case} together imply that $\mu_p(z)$ has a local minimum when $z$ lies on the geodesic $\gamma(p,q)$, and hence the derivative of $\mu_p$ with respect to $z$ is zero on this set.  By choosing $S$ sufficiently small, the derivative of $\mu_p(z)$ with respect to $z$ can be bounded by $1$, and hence the Lipschitz constant in $z$ will have this same bound.

\end{proof}

We will need to make use of the following theorem of Takeuchi \cite{Tak64} (see also \cite{CaSh05} for a simplified proof in the $C^2$ case):
\begin{thm}[Takeuchi]
  If $\Omega\subset\mathbb{CP}^n$ is a pseudoconvex domain, then there exists $E>0$ such that $i\ddbar(-\log\delta)\geq E\omega$ on $\Omega$ in the sense of currents.
\end{thm}

Let $\Gamma_p$, $\mu_p$, and $\phi_t^p$ be given by Lemma \ref{lem:tools}.  By Takeuchi's Theorem and Lemma \ref{lem:tools} \eqref{item:t_kahler_generator}, for any $p\in\mathbb{CP}^n$, $-\log\delta(z)-E\mu_p$ is plurisubharmonic when $z\in\Gamma_p$.

Set $\zeta=\frac{SA^3E}{2\sqrt{2}(1+A)^3B}$.  We may assume that $\zeta<1$ (if not, we can increase the size of $B$).  Choose $F>0$ and $0<\eta<\zeta$ so that
\begin{equation}
\label{eq:eta_defining_estimate}
  \eta\zeta^{-1} \left(-\log\left(\eta \zeta^{-1}\right)+1+F+2^{-1}E\log\sec S\right)
  <A.
\end{equation}
This is possible, since the lower bound in \eqref{eq:eta_defining_estimate} approaches $0$ as $\eta\rightarrow 0^+$.

For $p\in\partial\Omega$, $z\in \Gamma_p$, and $0\leq t<T-\delta(z)$, define
\[
  \lambda_t^p(z)=t^\eta\left(-\log(\eta t^{-1}\delta(\phi_t^p(z)))-\eta^{-1}+1+F+2^{-1}E(\log\sec S-\mu_p(z))\right).
\]
Since $\phi$ is holomorphic in $z$, $-\log\delta(\phi_t^p(z))-E\mu_p(\phi_t^p(z))$ will be plurisubharmonic. Using Lemma \ref{lem:tools} \eqref{item:t_kahler_generator} and Lemma \ref{lem:tools} \eqref{item:t_level_curves}, we have
\begin{equation}
\label{eq:plurisubharmonicity_bound}
  i\ddbar\lambda_t^p(z)\geq 2^{-1}t^\eta E\omega
\end{equation}
in the sense of currents.

Before optimizing in $t$, we will need to restrict our domain in $t$.  The following lemma gives us appropriate upper and lower bounds for $t$ based on $z$ and $p$.
\begin{lem}
\label{lem:a_b_defn}
  When $z\in\Omega$ and $p\in\partial\Omega$ satisfy $0<\delta(z)<\left(1-\zeta\right)T$ and $z\in\Gamma_p$, we may define
  \begin{align}
  \label{eq:a_b_defn}
    a_p(z)&=\inf\set{0<t<T-\delta(z):t^{-1}\delta(\phi_t^p(z))\leq\eta^{-1}},\\
    b_p(z)&=\inf\set{0<t<T-\delta(z):t^{-1}\delta(\phi_t^p(z))\leq\zeta^{-1}},
  \end{align}
  continuous in $z$ and $p$ satisfying
  \begin{equation}
  \label{eq:a_b_estimate}
    0<a_p(z)<b_p(z)<T-\delta(z).
  \end{equation}
  When $a_p(z)\leq t\leq b_p(z)$, we have
  \begin{equation}
  \label{eq:t_delta_estimate}
    \zeta^{-1}\leq t^{-1}\delta(\phi_t^p(z))\leq\eta^{-1}
  \end{equation}
\end{lem}

\begin{proof}
  From \eqref{eq:delta_derivative_estimate}, $D^+_t(t^{-1}\delta(\phi_t^p(z)))\leq D^-_t(t^{-1}\delta(\phi_t^p(z)))\leq t^{-1}-t^{-2}\delta(\phi_t^p(z))$, so $t^{-1}\delta(\phi_t^p(z))$ is decreasing in $t$ whenever $t^{-1}-t^{-2}\delta(\phi_t^p(z))\leq 0$, or $1\leq t^{-1}\delta(\phi_t^p(z))$.  Suppose $z\in\Gamma_p$ satisfies $0<\delta(z)<\left(1-\zeta\right)T$ and $0<t<T-\delta(z)$.  Since $\delta(z)>0$, $t^{-1}\delta(\phi_t^p(z))>\eta^{-1}$ for all $t$ sufficiently small.  By \eqref{eq:delta_estimate}, $t^{-1}\delta(\phi_t^p(z))\leq t^{-1}\delta(z)+1$.  When $t=T-\delta(z)$, we have $t^{-1}\delta(z)+1=\frac{T}{T-\delta(z)}<\zeta^{-1}$.  Hence, \eqref{eq:a_b_estimate} follows for $a_p$ and $b_p$ defined by \eqref{eq:a_b_defn}.  Since $t^{-1}\delta(\phi_t^p(z))$ is decreasing when it is larger than one, we also have \eqref{eq:t_delta_estimate} when $a_p(z)\leq t\leq b_p(z)$.  Observe that $a_p(z)$ and $b_p(z)$ must be continuous in $z$ since $\phi_t^p(z)$ is continuous in $z$ and continuous in $p$ by Lemma \ref{lem:tools} \eqref{item:t_Lipschitz_in_p}.
\end{proof}

We are now ready to optimize in $t$ and obtain a local construction for our weight function on $\Gamma_p$.  For any $p\in\partial\Omega$, define
\[
  \dom\lambda^p=\set{z\in\Omega:z\text{ is in the interior of }\Gamma_p\text{ and }0<\delta(z)<\left(1-\zeta\right)T}.
\]
For $z\in\dom\lambda^p$, we define
\begin{equation}
\label{eq:lambda_p_defn}
  \lambda^p(z)=\sup_{a_p(z)\leq t\leq b_p(z)}\lambda_t^p(z).
\end{equation}
Our first goal is to show that the supremum is obtained for $a_p(z)<t<b_p(z)$.

\begin{lem}
\label{lem:interior_t}
  For $p\in\partial\Omega$ and $z\in\dom\lambda^p$, let $\lambda^p(z)$ be defined by \eqref{eq:lambda_p_defn}.  Then $\lambda^p(z)>\lambda_{a_p(z)}^p(z)$ and $\lambda^p(z)>\lambda_{b_p(z)}^p(z)$.
\end{lem}

\begin{proof}
  Suppose $t=a_p(z)$.  From \eqref{eq:delta_derivative_estimate}, we have
  \[
    D_t^+\lambda_{t}^p(z)\geq \eta t^{-1}\lambda_t^p(z)+t^{\eta-1}-\frac{t^\eta}{\delta(\phi_t^p(z))}.
  \]
  When $t=a_p(z)$, we have $\delta(\phi_t^p(z))=t\eta^{-1}$, so on $\Gamma_p$
  \[
    \lambda_t^p(z)=t^\eta\left(-\eta^{-1}+1+F+2^{-1}E(\log\sec S-\mu_p(z))\right)\geq t^\eta\left(-\eta^{-1}+1+F\right).
  \]
  Hence,
  \[
    \left.D_t^+\lambda_{t}^p(z)\right|_{t=a_p(z)}\geq \eta t^{\eta-1}F>0,
  \]
  so $\lambda_t^p(z)$ is not maximized on $[a_p(z),b_p(z)]$ when $t=a_p(z)$.

  On the other side, suppose $t=b_p(z)$.  This time, \eqref{eq:delta_derivative_estimate} will give us
  \[
    D^-_t\lambda_t^p(z)\leq \eta t^{-1}\lambda_t^p(z)+t^{\eta-1}-\frac{At^\eta}{\delta(\phi_t^p(z))}.
  \]
  Since $t=b_p(z)$ implies $\delta(\phi_t^p(z))=\zeta^{-1}t$, we have
  \[
    \lambda_t^p(z)\leq t^\eta\left(-\log\left(\eta \zeta^{-1}\right)-\eta^{-1}+1+F+2^{-1}E\log\sec S\right),
  \]
  so
  \[
    \left.D^-_t\lambda_t^p(z)\right|_{t=b_p(z)}\leq \eta t^{\eta-1}\left(-\log\left(\eta \zeta^{-1}\right)+1+F+2^{-1}E\log\sec S-A\zeta\eta^{-1}\right).
  \]
  This is strictly negative by \eqref{eq:eta_defining_estimate}, so $\lambda_t^p(z)$ is not maximized on $[a_b(z),b_p(z)]$ when $t=b_p(z)$.
\end{proof}

With this, we can now show that $\lambda_p$ has the necessary local properties.  We will see that the following lemma gives us a local defining function $-(-\lambda_p)^{1/\eta}$ satisfying the requirements of Theorem \ref{thm:main_theorem}.
\begin{lem}
\label{lem:lambda_p_continuity}
  For $p\in\partial\Omega$ let $\lambda^p$ be defined by \eqref{eq:lambda_p_defn}.  Then on $\dom\lambda^p$, $\lambda^p$ is Lipschitz continuous with
  \begin{equation}
  \label{eq:lambda_p_Lipschitz}
    \abs{\nabla\lambda^p}\leq O((\delta(z))^{\eta-1})
  \end{equation}
  almost everywhere uniformly in $p$ and $\lambda_p$ is strictly plurisubharmonic with
  \begin{equation}
  \label{eq:lambda_p_plurisubharmonic}
    i\ddbar\lambda^p\geq 2^{-1}E\left(\eta^{-1}-A\right)^{-\eta}(\delta(z))^\eta\omega
  \end{equation}
  in the sense of currents.
\end{lem}

\begin{proof}
  Since $\phi_t^p(z)$ is an isometry, it has a Lipschitz constant of $1$.  By Lemma \ref{lem:tools} \eqref{item:t_Lipschitz_in_z}, the Lipschitz constant of $\lambda_t^p(z)$ with respect to $z$ is locally bounded by $t^\eta\left((\delta(\phi_t^p(z)))^{-1}+2^{-1}E\right)$.  By \eqref{eq:delta_estimate} and \eqref{eq:t_delta_estimate} when $a_p(z)\leq t\leq b_p(z)$ we have,
  \[
    \delta(z)\leq\delta(\phi_t^p(z))-At\leq\min\set{\left(\eta^{-1}-A\right)t,\left(1-A\eta\right)\delta(\phi_t^p(z))}
  \]
  and
  \[
    \delta(z)\geq\delta(\phi_t^p(z))-t\geq\max\set{\left(\zeta^{-1}-1\right)t,\left(1-\zeta\right)\delta(\phi_t^p(z))},
  \]
  so
  \begin{equation}
  \label{eq:bounds_t}
    \left(\eta^{-1}-A\right)^{-1}\delta(z)\leq t\leq\left(\zeta^{-1}-1\right)^{-1}\delta(z),
  \end{equation}
  and
  \begin{equation}
  \label{eq:bounds_delta_phi}
    \left(1-A\eta\right)^{-1}\delta(z)\leq\delta(\phi_t^p(z))\leq\left(1-\zeta\right)^{-1}\delta(z).
  \end{equation}
  Hence, when $a_p(z)\leq t\leq b_p(z)$, $t$ and $\delta(\phi_t^p(z))$ are both uniformly comparable with $\delta(z)$, so the Lipschitz constant of $\lambda_t^p(z)$ with respect to $z$ is locally uniformly bounded by terms on the order of $O\left((\delta(z))^{\eta-1}\right)$.

  For $z\in\dom\lambda^p$, fix $a_p(z)<t<b_p(z)$ such that $\lambda^p(z)=\lambda^p_t(z)$.  Suppose that for every $j\in\mathbb{N}$ there exists $z_j\in B(z,1/j)\cap\dom\lambda^p$ and $a_p(z_j)<t_j<b_p(z_j)$ such that $\lambda^p(z_j)=\lambda^p_{t_j}(z_j)$ and either $t\geq b_p(z_j)$, $t\leq a_p(z_j)$, $t_j\geq b_p(z)$, or $t_j\leq a_p(z)$.  We will show that each of these leads to a contradiction for sufficiently large $j$.  Since $b_p$ is continuous in $z$ and $[z_j]\rightarrow[z]$, $t\geq b_p(z_j)$ would imply that $t\geq b_p(z)$, contradicting Lemma \ref{lem:interior_t}.  Hence $t<b_p(z_j)$ for all sufficiently large $j$, and the same argument shows that $t>a_p(z_j)$ for all sufficiently large $j$.  If $t_j\geq b_p(z)$ for infinitely many $j$, then there must be a limit point $s$ of $\set{t_j}$ satisfying $s\geq b_p(z)$.  Since $t_j\leq b_p(z_j)$, continuity again implies $s\leq b_p(z)$, so $s=b_p(z)$.  However, Lemma \ref{lem:interior_t} implies that $\lambda_t^p(z)>\lambda_s^p(z)$, so for sufficiently large $j$ we must have $\lambda_t^p(z_j)>\lambda_{t_j}^p(z_j)$, a contradiction.  Since $t_j\leq a_p(z)$ leads to a similar contradiction, we conclude that there must exist $R>0$ such that for every $w\in B(z,R)\cap\dom\lambda^p$ and $a_p(w)<s<b_p(w)$ satisfying $\lambda^p(w)=\lambda^p_s(w)$ we have $a_p(z)<s<b_p(z)$ and $a_p(w)<t<b_p(w)$.  We may assume that $B(z,R)\subset\dom\lambda^p$ and $\lambda_t^p(w)$ has Lipschitz constant bounded by $M(\delta(z))^{\eta-1}$ on $B(z,R)$ for all $a_p(z)<t<b_p(z)$ with $M$ independent of $z$ and $R$.

  Now, for $w\in B(z,R)$ we have
  \[
    \lambda^p(w)=\lambda^p_s(w)\leq\lambda^p_s(z)+M(\delta(z))^{\eta-1}\dist(z,w)\leq\lambda^p(z)+M(\delta(z))^{\eta-1}\dist(z,w)
  \]
  and
  \[
    \lambda^p(z)=\lambda^p_t(z)\leq\lambda^p_t(w)+M(\delta(z))^{\eta-1}\dist(z,w)\leq\lambda^p(w)+M(\delta(z))^{\eta-1}\dist(z,w),
  \]
  so $\abs{\lambda^p(z)-\lambda^p(w)}\leq M(\delta(z))^{\eta-1}$ on $B(z,R)$, and \eqref{eq:lambda_p_Lipschitz} follows.  By applying the sub-mean-value property on $B(z,R)$ with \eqref{eq:plurisubharmonicity_bound}, we may use \eqref{eq:bounds_t} to show \eqref{eq:lambda_p_plurisubharmonic}.

\end{proof}

Now, we are ready to optimize in $p$ and obtain a global plurisubharmonic function.  Let
\[
  \dom\lambda=\set{z\in\Omega:0<\delta(z)<\left(1-\zeta\right)T}.
\]
For any $z\in\dom\lambda$, we define
\begin{equation}
\label{eq:lambda_defn}
  \lambda(z)=\sup_{\set{p\in\partial\Omega:z\in\Gamma_p}}\lambda^p(z).
\end{equation}
Although $\lambda^p$ can still be defined for $z\in\overline{\dom\lambda^p}$, we lose the important properties of Lemma \ref{lem:lambda_p_continuity}.  Since Lemma \ref{lem:tools} \eqref{item:t_continuity} and Lemma \ref{lem:tools} \eqref{item:t_gamma_characterization} imply that this supremum is taken over a compact set, we must have $\lambda(z)=\lambda^p(z)$ for some $p\in\partial\Omega$.  Our next goal is to show that $z$ is in the interior of $\Gamma_p$ so that $\lambda^p$ is Lipschitz and plurisubharmonic in a neighborhood of $z$.

\begin{lem}
\label{lem:interior_Gamma}
  For $z\in\dom\lambda$, define $\lambda$ by \eqref{eq:lambda_defn}.  Then for any $p\in\partial\Omega$ satisfying $z\in\partial\Gamma_p$, $\lambda(z)>\lambda^p(z)$.
\end{lem}

\begin{proof}
  Suppose that $z$ is on the boundary of $\Gamma_p$.  Then we must have $\mu_p(z)=\log\sec S$ by Lemma \ref{lem:tools} \eqref{item:t_gamma_characterization}.  Let $u\in\Tan(\partial\Omega,p)$ be given by Lemma \ref{lem:tools} \eqref{item:t_boundary_sequence} and let $\{p_j\}\subset\partial\Omega$ satisfy $\gamma_0'(p,p_j)\rightarrow u$.  Fix $a_p(z)<t<b_p(z)$ so that $\lambda^p(z)=\lambda_t^p(z)$.  Since $a_p$ and $b_p$ are continuous in $p$, we must have $a_{p_j}(z)<t<b_{p_j}(z)$ for $j$ sufficiently large.  Then
  \begin{multline*}
    \lambda_t^{p_j}(z)-\lambda_t^p(z)=t^\eta\left(-\log\left(\frac{\delta(\phi_t^{p_j}(z))}{\delta(\phi_t^p(z))}\right)-2^{-1}E(\mu_{p_j}(z)-\mu_p(z))\right)\\
    \geq t^\eta\left(1-\frac{\delta(\phi_t^{p_j}(z))}{\delta(\phi_t^p(z))}-2^{-1}E(\mu_{p_j}(z)-\mu_p(z))\right).
  \end{multline*}
  Since $\delta$ always has Lipschitz constant one, \eqref{eq:Lipschitz_in_p} gives us $\abs{1-\frac{\delta(\phi_t^{p_j}(z))}{\delta(\phi_t^p(z))}}\leq\frac{Bt\dist(p,p_j)}{\delta(\phi_t^p(z))}$.  Hence, Lemma \ref{lem:tools} \eqref{item:t_boundary_sequence} gives us
  \begin{multline*}
    \liminf_{j\rightarrow\infty}\frac{\lambda_t^{p_j}(z)-\lambda_t^p(z)}{\dist(p,p_j)}\\
    \geq t^\eta\left(-\frac{Bt}{\delta(\phi_t^p(z))}+\frac{SA^3E}{2\sqrt{2}(1+A)^3}\right)=B t^\eta\left(-\frac{t}{\delta(\phi_t^p(z))}+\zeta\right).
  \end{multline*}
  By \eqref{eq:t_delta_estimate}, $t\leq b_p(z)$ implies $\liminf_{j\rightarrow\infty}\frac{\lambda_t^{p_j}(z)-\lambda_t^p(z)}{\dist(p,p_j)}>0$.  Hence, for infinitely many $p_j$, $\lambda_t^{p_j}(z)>\lambda_t^p(z)$, so we conclude that
  \[
    \lambda(z)\geq\lambda^{p_j}(z)\geq\lambda^{p_j}_t(z)>\lambda_t^p(z)=\lambda^p(z)
  \]
  when $z$ lies on the boundary of $\Gamma_p$.
\end{proof}

With this, we can demonstrate that $\lambda$ inherits the necessary properties from $\lambda_p$.
\begin{lem}
\label{lem:lambda_continuity}
  Let $\lambda$ be defined by \eqref{eq:lambda_defn}.  Then on $\dom\lambda$, $\lambda$ is Lipschitz continuous with
  \begin{equation}
  \label{eq:lambda_Lipschitz}
    \abs{\nabla\lambda}\leq O((\delta(z))^{\eta-1})
  \end{equation}
  almost everywhere and strictly plurisubharmonic with
  \begin{equation}
  \label{eq:lambda_plurisubharmonic}
    i\ddbar\lambda\geq 2^{-1}E\left(\eta^{-1}-A\right)^{-\eta}(\delta(z))^\eta\omega
  \end{equation}
  in the sense of currents.
\end{lem}

\begin{proof}
  For $z\in\dom\lambda$, Lemma \ref{lem:interior_Gamma} implies that there exists $p\in\partial\Omega$ such that $z\in\dom\lambda^p$ and $\lambda^p(z)=\lambda(z)$.  Suppose that for every $j\in\mathbb{N}$, there exists $z_j\in\dom\lambda\cap B(z,1/j)$ and $p_j\in\partial\Omega$ such that $z_j\in\dom\lambda^{p_j}$, $\lambda^{p_j}(z_j)=\lambda(z_j)$, and $z\notin\dom\lambda^{p_j}$.  If $q$ is any limit point of $\set{p_j}$, then we must have $z\in\Gamma_q$ by Lemma \ref{lem:tools} \eqref{item:t_continuity} and Lemma \ref{lem:tools} \eqref{item:t_gamma_characterization}.  Lemma \ref{lem:interior_Gamma} implies $\lambda(z)>\lambda^q(z)$.  However, this implies $\lambda(z_j)>\lambda^{p_j}(z_j)$ for $j$ sufficiently large, so we have a contradiction.

  Hence, there exists $R>0$ so that for every $w\in\dom\lambda\cap B(z,R)$ and $q\in\partial\Omega$ satisfying $w\in\dom\lambda^q$ and $\lambda^q(w)=\lambda(w)$, $z\in\dom\lambda^q$.  We may assume that $B(z,R)\subset\dom\lambda^p$.  Using \eqref{eq:lambda_p_Lipschitz}, we may also assume that the Lipschitz constant of $\lambda^q$ is bounded by $M(\delta(z))^{\eta-1}$ whenever $z\in\dom\lambda^q$.  Hence,
  \[
    \lambda(w)=\lambda^q(w)\leq\lambda^q(z)+M(\delta(z))^{\eta-1}\dist(z,w)\leq\lambda(z)+M(\delta(z))^{\eta-1}\dist(z,w)
  \]
  and
  \[
    \lambda(z)=\lambda^p(z)\leq\lambda^p(w)+M(\delta(z))^{\eta-1}\dist(z,w)\leq\lambda(w)+M(\delta(z))^{\eta-1}\dist(z,w),
  \]
  so $\abs{\lambda(z)-\lambda(w)}\leq M(\delta(z))^{\eta-1}\dist(z,w)$ for $w\in B(z,R)$, and \eqref{eq:lambda_Lipschitz} follows.  We may apply the sub-mean-value property on $B(z,R)$ with \eqref{eq:lambda_p_plurisubharmonic} to show \eqref{eq:lambda_plurisubharmonic}.
\end{proof}

To estimate $\lambda$, we observe that by \eqref{eq:t_delta_estimate} we have
\[
  \lambda_t^p(z)\leq t^\eta\left(-\log\left(\eta \zeta^{-1}\right)-\eta^{-1}+1+F+2^{-1}E\log\sec S\right).
\]
From \eqref{eq:eta_defining_estimate}, this implies
\[
  \lambda_t^p(z)\leq \eta^{-1} t^\eta\left(A\zeta-1\right).
\]
Using \eqref{eq:bounds_t},
\[
  \lambda_t^p(z)\leq -\eta^{-1} (\eta^{-1}-A)^{-\eta}\left(1-A\zeta\right)(\delta(z))^\eta.
\]
On the other hand, \eqref{eq:t_delta_estimate} also gives us
\[
  \lambda_t^p(z)\geq t^\eta\left(-\eta^{-1}+1+F\right),
\]
so \eqref{eq:bounds_t} implies
\[
  \lambda_t^p(z)\geq -\left(\zeta^{-1}-1\right)^{-\eta}\left(\eta^{-1}-1-F\right)(\delta(z))^\eta.
\]
Since these upper and lower bounds are independent of $t$ and $p$, we conclude
\begin{multline*}
  -\left(\zeta^{-1}-1\right)^{-\eta}\left(\eta^{-1}-1-F\right)(\delta(z))^\eta\\
  \leq\lambda(z)\leq-\eta^{-1} (\eta^{-1}-A)^{-\eta}\left(1-A\zeta\right)(\delta(z))^\eta.
\end{multline*}

Choose $T_1$ and $\lambda_1$ satisfying
\begin{gather*}
  0<T_1<\left(1-\zeta\right)T,\\
  0>\lambda_1>-\eta^{-1} (\eta^{-1}-A)^{-\eta}\left(1-A\zeta\right)T_1^\eta,
\end{gather*}
so that $\lambda(z)<\lambda_1$ when $\delta(z)=T_1$.  We have $\lambda(z)>\lambda_1$ whenever $\delta(z)<\left(\zeta^{-1}-1\right)\left(\eta^{-1}-1-F\right)^{-1/\eta}(-\lambda_1)^{1/\eta}$, so we may choose $T_0$ and $\lambda_0$ satisfying
\begin{gather*}
  0<T_0<\left(\zeta^{-1}-1\right)\left(\eta^{-1}-1-F\right)^{-1/\eta}(-\lambda_1)^{1/\eta}<T_1,\\
  \lambda_1<\lambda_0<-\left(\zeta^{-1}-1\right)^{-\eta}\left(\eta^{-1}-1-F\right)T_0^\eta,
\end{gather*}
so that $\lambda(z)>\lambda_0$ when $\delta(z)=T_0$.  Now,
\[
  \psi(z)=-\frac{\lambda_0-\lambda_1}{\log(T_1/T_0)}\log\delta(z)+\frac{\lambda_0\log T_1-\lambda_1\log T_0}{\log(T_1/T_0)}
\]
is a strictly plurisubharmonic function on $\Omega$ satisfying $\psi(z)=\lambda_1$ when $\delta(z)=T_1$ and $\psi(z)=\lambda_0$ when $\delta(z)=T_0$, so we can define a Lipschitz continuous strictly plurisubharmonic function $\tilde\lambda$ via
\[
  \tilde\lambda(z)=\begin{cases}
    \lambda(z)&\delta(z)\leq T_0,\\
    \max\set{\lambda(z),\psi(z)}& T_0<\delta(z)<T_1,\\
    \psi(z)&\delta(z)\geq T_1.
  \end{cases}
\]
If we define
\[
  \rho(z)=\begin{cases}-(-\tilde\lambda(z))^{1/\eta}&z\in\Omega\\\delta(z)&z\notin\Omega\end{cases},
\]
then we have the Lipschitz defining function satisfying the requirements of Theorem \ref{thm:main_theorem}.

\section{A Counterexample}
\label{sec:counter_example}

We will work in a local coordinate patch $\{(\tilde{z}_1,\ldots,\tilde{z}_n)\}$, with the usual convention that $\tilde{z}'=(\tilde{z}_1,\ldots,\tilde{z}_{n-1})$.  In this coordinate patch, we have $\dist(\tilde{z},\tilde{w})=\arccos\left(\frac{\abs{\tilde{z}\cdot\overline{\tilde{w}}+1}}{\sqrt{\abs{\tilde{z}}^2+1}\sqrt{\abs{\tilde{w}}^2+1}}\right)$.  Consider the domain
\[
  \Omega=\set{\tilde{z}\in\mathbb{C}^n:\abs{\tilde{z}}<1\text{ and either }\re \tilde{z}_n<0\text{ or }\im\tilde{z}_n<0}.
\]
The set
\[
  \set{\tilde{z}\in\mathbb{C}^n:\text{either }\re\tilde{z}_n\leq 0\text{ and }\im\tilde{z}_n=0\text{ or }\re\tilde{z}_n=0\text{ and }\im\tilde{z}_n\leq 0}
\]
is foliated by complex hypersurfaces, so it must be the Levi-flat boundary of a pseudoconvex domain.  Therefore $\Omega$ is the intersection of pseudoconvex domains, so it is also pseudoconvex.  The intersection is transverse (in the real sense), so $\Omega$ must also be Lipschitz.

On the set
\[
  \mathcal{O}=\set{\tilde{z}\in\mathbb{C}^n:\abs{\tilde{z}}<\sqrt{2}-1\text{, }\re \tilde{z}_n<0\text{, and }\im\tilde{z}_n<0},
\]
the point on $\partial\Omega$ minimizing the distance to $\tilde{z}$ is $(\tilde{z}',0)$, so
\[
  \delta(\tilde{z})|_{\mathcal{O}}=\arccos\left(\frac{\sqrt{\abs{\tilde{z}'}^2+1}}{\sqrt{\abs{\tilde{z}}^2+1}}\right)=\arcsin\left(\frac{\abs{\tilde{z}_n}}{\sqrt{\abs{\tilde{z}}^2+1}}\right)
\]
When $\tilde{z}=(0,\ldots,0,\tilde{z}_n)$, we have $\delta(\tilde{z})=\arccos\left(\frac{1}{\sqrt{\abs{\tilde{z}_n}^2+1}}\right)$, so
\[
  \frac{\partial\delta}{\partial\tilde{z}_n}(\tilde{z})=\frac{\overline{\tilde{z}_n}}{2\abs{\tilde{z}_n}(\abs{\tilde{z}_n}^2+1)}
\]
and
\[
  \frac{\partial^2\delta}{\partial\tilde{z}_n\partial\overline{\tilde{z}_n}}(\tilde{z})=\frac{1-\abs{\tilde{z}_n}^2}{4\abs{\tilde{z}_n}(\abs{\tilde{z}_n}^2+1)^2}.
\]
Then
\[
  \frac{\partial^2(-\delta^\eta)}{\partial\tilde{z}_n\partial\overline{\tilde{z}_n}}(\tilde{z})=-\eta\delta^{\eta-1}(\tilde{z})\left(\frac{1-\abs{\tilde{z}_n}^2}{4\abs{\tilde{z}_n}(\abs{\tilde{z}_n}^2+1)^2}\right)+\eta(1-\eta)\delta^{\eta-2}(\tilde{z})\frac{1}{4(\abs{\tilde{z}_n}^2+1)^2},
\]
so if $-\delta^\eta(\tilde{z})$ is plurisubharmonic we have $\delta(\tilde{z})\leq \frac{(1-\eta)\abs{\tilde{z}_n}}{1-\abs{\tilde{z}_n}^2}$.  However, since $\arcsin x$ is convex when $x>0$, we have $\delta(\tilde{z})\geq\frac{\abs{\tilde{z}_n}}{\sqrt{\abs{\tilde{z}_n}^2+1}}$, so we must have $\frac{1}{\sqrt{\abs{\tilde{z}_n}^2+1}}\leq\frac{1-\eta}{1-\abs{\tilde{z}_n}^2}$.  Letting $\tilde{z}_n\rightarrow 0$, we have $\eta\leq 0$, a contradiction.  Hence, Proposition \ref{prop:counterexample} follows.
% ----------------------------------------------------------------
\bibliographystyle{amsplain}
\bibliography{harrington}
\end{document}